\documentclass[reqno]{amsart}
\usepackage{amsfonts,amsmath,amssymb}

\numberwithin{equation}{section}

\newtheorem{thm}{Theorem}[section]
\newtheorem{lem}{Lemma}[section]

\newtheorem{cor}{Corollary}[section]

\newcommand{\fincal}{\end{eqnarray*}}

\begin{document}
\title[EKGM systems on compact Riemannian manifolds]
{Existence and a priori bounds for\\
electrostatic Klein-Gordon-Maxwell\\
systems in fully inhomogeneous spaces}
\author{Olivier Druet}
\address{Olivier Druet, Ecole normale sup\'erieure de Lyon, D\'epartement de 
Math\'ematiques - UMPA, 46 all\'ee d'Italie, 69364 Lyon cedex 07,
France}
\email{Olivier.Druet@umpa.ens-lyon.fr}
\author{Emmanuel Hebey}
\address{Emmanuel Hebey, Universit\'e de Cergy-Pontoise, 
D\'epartement de Math\'ematiques, Site de 
Saint-Martin, 2 avenue Adolphe Chauvin, 
95302 Cergy-Pontoise cedex, 
France}
\email{Emmanuel.Hebey@math.u-cergy.fr}

\thanks{To appear in Communications in Contemporary Mathematics\\
The authors were partially supported by the ANR grant 
ANR-08-BLAN-0335-01. The second author was also partially supported by a MIT-France fund.}

\begin{abstract} We prove existence and uniform bounds  
for electrostatic Klein-Gordon-Maxwell systems in the inhomogeneous 
context of a compact Riemannian manifold when the mass potential, balanced by the phase, 
is small in a quantified sense. Phase compensation for electrostatic Klein-Gordon-Maxwell systems 
and the positive mass theorem are used in a crucial way. 
\end{abstract}

\maketitle

Electrostatic Klein-Gordon-Maxwell systems are electrostatic derivations of the Klein-Gordon-Maxwell systems which, in turn, are a special case of the Yang-Mills-Higgs 
equation. They arise naturally in quantum mechanics. Roughly speaking, 
Klein-Gordon-Maxwell systems provide a dualistic model for the description of 
the interaction between a charged relativistic particle of matter and the electromagnetic field that it generates. The electromagnetic field is both generated by and drives the 
particle field. In the electrostatic form of the Klein-Gordon-Maxwell systems, writing the matter particle as a standing wave $ue^{i\omega t}$,  
it is characterized by the property that 
$u$ solves the electrostatic Klein-Gordon-Maxwell systems we investigate in this 
paper with a gauge potential $v$. 
In what follows we let $(M,g)$ be a smooth compact $3$-dimensional Riemannian manifold and $a > 0$ be a smooth positive 
function in $M$. Let $\omega_a$ be given by
\begin{equation}\label{DefOm0a}
\omega_a = \sqrt{\min_Ma}\hskip.1cm .
\end{equation}
Given real numbers $q > 0$, $\omega \in (-\omega_a,\omega_a)$, $\lambda \ge 0$, and $p \in (2,6]$, the 
electrostatic Klein-Gordon-Maxwell systems we investigate 
in this paper are written as
\begin{equation}\label{SWSyst}
\begin{cases}
\Delta_gu + au = u^{p-1} + \omega^2\left(qv-1\right)^2u\\
\Delta_gv + \left(\lambda+q^2u^2\right)v = qu^2\hskip.1cm ,
\end{cases}
\end{equation}
where $\Delta_g = -\hbox{div}_g\nabla$ is the Laplace-Beltrami operator. The system \eqref{SWSyst} is energy critical when $p = 6$, and 
subcritical when $p \in (2,6)$. In the 
classical physical setting, 
$a = m_0^2$, $m_0$ is the mass of the particle, the coercivity constant $\lambda = 0$, $q$ is the charge of the particle, $u$ is the field associated to the particle, 
$\omega$ is the temporal frequency (referred to as the phase in the sequel), and $v$ is the electric potential. The constant $\lambda$ in this paper can be interpreted in terms 
of the Maxwell-Proca theory (see Hebey and Truong \cite{HebTru}). 
In what follows we let $S_g$ stand for the scalar curvature of $g$. Also we let $\mathcal{S}_p(\omega)$ be 
the set consisting of the positive smooth solutions $\mathcal{U} = (u,v)$ 
of \eqref{SWSyst} with phase $\omega$ and nonlinear term $u^{p-1}$. Namely,
\begin{equation}\label{DefSp}
\mathcal{S}_p(\omega) = \Bigl\{(u,v)\hskip.1cm\hbox{smooth}\hskip.1cm\hbox{s.t.}\hskip.1cm u > 0, v > 0,\hskip.1cm\hbox{and}\hskip.1cm (u,v)
\hskip.1cm\hbox{solve}\hskip.1cm\eqref{SWSyst}\Bigr\}\hskip.1cm .
\end{equation}
Given $\omega \in [0,\omega_a)$, we let
\begin{equation}\label{DefKOm0}
K_0(\omega) = (-\omega_a,-\omega]\bigcup [\omega,\omega_a)\hskip.1cm, 
\end{equation}
where $\omega_a$ is as in \eqref{DefOm0a}. When $\omega = 0$, $K_0(0)$ is the interval 
$K_\varepsilon(0) = (-\omega_a,\omega_a)$. 
For $\theta \in (0,1)$, and $\mathcal{U} = (u,v)$, we let 
\begin{equation}\label{C2ThetaNorm}
\Vert\mathcal{U}\Vert_{C^{2,\theta}} = \Vert u\Vert_{C^{2,\theta}} + \Vert v\Vert_{C^{2,\theta}}\hskip.1cm .
\end{equation}
We recall that $(M,g)$ is said to be conformally diffeomorphic to the unit $3$-sphere $(S^3,g_0)$ 
if there exists a diffeomorphism $\varphi: S^3 \to M$ such that 
$\varphi^\star g = u^4g_0$ for some smooth positive function $u$ in $S^3$. 
We prove below the existence of smooth positive solutions and 
the existence of uniform bounds for \eqref{SWSyst} in the subcritical cases $p \in (2,6)$ without any conditions, and in the critical case $p =6$ 
assuming that the mass potential, balanced by the phase, 
is smaller than the geometric threshold potential of the conformal Laplacian. Our main result is as follows. Closely related estimates are derived 
in Theorem \ref{SecThm} in Section \ref{SecRes}.

\begin{thm}\label{MainThm} Let $(M,g)$ be a smooth compact $3$-dimensional Riemannian manifold and $a> 0$ be a smooth positive 
function in $M$. Let $q > 0$, 
$\omega \in (-\omega_a,\omega_a)$, $\lambda \ge 0$, and $p \in (2,6]$, where $\omega_a$ is as in \eqref{DefOm0a}. When $p = 6$ assume
\begin{equation}\label{MainAssumpt}
a \le k\lambda\omega^2 + \frac{1}{8}S_g
\end{equation}
in $M$ for some $k  > 0$ such that $k\lambda < 1$. Then \eqref{SWSyst} possesses a smooth positive solution. 
Moreover, for any $p \in (2,6)$, and any $\theta \in (0,1)$, there exists $C >0$ such that for any $\omega^\prime \in K_0(0)$, and 
any $\mathcal{U} \in \mathcal{S}_p(\omega^\prime)$, $\Vert\mathcal{U}\Vert_{C^{2,\theta}} \le C$, where $\mathcal{S}_p(\omega^\prime)$ 
is as in \eqref{DefSp}, $K_0(0)$ is as in \eqref{DefKOm0}, and $\Vert\cdot\Vert_{C^{2,\theta}}$ 
is as in \eqref{C2ThetaNorm}. Assuming again \eqref{MainAssumpt}, 
with the property that \eqref{MainAssumpt} is strict at least at one point if $(M,g)$ is conformally diffeomorphic to the unit $3$-sphere 
and $\omega\lambda = 0$, there also holds that for any $\theta \in (0,1)$, 
$\Vert\mathcal{U}\Vert_{C^{2,\theta}} \le C$ for all $\mathcal{U} \in \mathcal{S}_6(\omega^\prime)$ 
and all $\omega^\prime \in K_0(\omega)$, where $C > 0$ does not depend on $\omega^\prime$ and $\mathcal{U}$.
\end{thm}

There are several consequences to our theorem. The first obvious one is that solutions in the subcritical case exist for all phases and are uniformly bounded 
in $C^{2,\theta}$. For the sake of clearness we restate this result in the following corollary.

\begin{cor}[Subcritical Case]\label{Cor1} Let $(M,g)$ be a smooth compact $3$-dimensional Riemannian manifold and 
$a > 0$ be a smooth positive function in $M$. Let $q > 0$, $\lambda \ge 0$, and $p \in (2,6)$. 
For any $\omega \in (-\omega_a,\omega_a)$ there exists a smooth positive solution of 
\eqref{SWSyst}. Moreover, for any $\theta \in (0,1)$, there exists $C > 0$ such that 
$\Vert\mathcal{U}\Vert_{C^{2,\theta}} \le C$
for all $\omega \in (-\omega_a,\omega_a)$, and all $\mathcal{U} \in \mathcal{S}_p(\omega)$.
\end{cor}

Two notable consequences of the theorem concern the critical case where it holds that $p = 6$. Assuming $\lambda > 0$, 
the first consequence we discuss, which provides a perfect illustration of phase compensation, 
is that if the oscillation of $a$, given by $\hbox{Osc}(a) = \max_Ma-\min_Ma$, is not too large, 
then there always are solutions of our system for sufficiently large phases and such solutions are again uniformly bounded in 
$C^{2,\theta}$. 

\begin{cor}[Critical Case 1]\label{Cor2} Let $(M,g)$ be a smooth compact $3$-dimensional Riemannian manifold and 
$a > 0$ be a smooth positive function in $M$. Let $q > 0$ and $\lambda > 0$. Suppose $\hbox{Osc}(a) < \frac{1}{8}\min_MS_g$. Then 
there exists $\varepsilon > 0$ such that for any $\omega \in K_0(\omega_a-\varepsilon)$, 
\eqref{SWSyst} possesses a smooth positive solution when $p = 6$. Moreover, 
for any $\theta \in (0,1)$, there exists $C > 0$ such that 
$\Vert\mathcal{U}\Vert_{C^{2,\theta}} \le C$
for all $\omega \in K_0(\omega_a-\varepsilon)$, and all $\mathcal{U} \in \mathcal{S}_6(\omega)$.
\end{cor}

The last consequence of the theorem we discuss, still dealing with the critical case where $p = 6$, 
concerns the more restrictive case where $a\le \frac{1}{8}S_g$. In this case, when $a$ is not too large, 
we get that there are solutions for all phases and that such solutions are uniformly bounded in 
$C^{2,\theta}$ for all phases.

\begin{cor}[Critical Case 2]\label{Cor3} Let $(M,g)$ be a smooth compact $3$-dimensional Riemannian manifold and 
$a > 0$ be a smooth positive function in $M$. Let $q > 0$ and $\lambda \ge 0$. Suppose $a \le \frac{1}{8}S_g$, 
the inequality being strict at least at one point in case the manifold is conformally diffeomorphic to the unit $3$-sphere. 
For any $\omega \in (-\omega_a,\omega_a)$ there exists a smooth positive solution of 
\eqref{SWSyst} when $p = 6$. Moreover, for any $\theta \in (0,1)$, there exists $C > 0$ such that 
$\Vert\mathcal{U}\Vert_{C^{2,\theta}} \le C$
for all $\omega \in (-\omega_a,\omega_a)$, and all $\mathcal{U} \in \mathcal{S}_6(\omega)$.
\end{cor}

\medskip As an immediate consequence of the $C^{2,\theta}$-bounds in the above results 
we obtain phase stability for standing waves of the Klein-Gordon-Maxwell equations in electrostatic form. 
Standing waves for the Klein-Gordon-Maxwell equations in electrostatic form can 
be written as $S = ue^{i\omega t}$ and they are coupled with a gauge potential $v$, where 
$(u,v)$ solves \eqref{SWSyst}. Roughly speaking, phase stability means that for arbitrary sequences 
of standing waves $u_\alpha e^{i\omega_\alpha t}$, with gauge potentials $v_\alpha$, the convergence of the 
phase $\omega_\alpha$ in $\mathbb{R}$ implies the convergence of the amplitude $u_\alpha$ and of the gauge $v_\alpha$ 
in the $C^2$-topology. 
In the subcritical case, it follows from Corollary \ref{Cor1} that 
for any sequence 
$(\omega_\alpha)_\alpha$, $\alpha \in \mathbb{N}$, 
and for any sequence of standing waves $(x,t) \to u_\alpha(x)e^{i\omega_\alpha t}$, 
with gauge potentials $v_\alpha$, 
if $\omega_\alpha \to \omega$ 
as $\alpha \to +\infty$ and $\vert\omega\vert < \omega_a$, then, up to a subsequence, 
$u_\alpha \to u$ and $v_\alpha \to v$ in $C^2$ as $\alpha \to +\infty$ for some smooth functions $u$ and $v$. By Lemma 
 \ref{LemNonZer}, $u$ and $v$ are positive and they give rise to another standing wave $ue^{i\omega t}$ with gauge potential $v$. 
 In particular, phase stability in the 
subcritical case holds true without any condition. By Corollary \ref{Cor3}, phase stability remains true in the 
critical case where $p = 6$ if we assume that $a\le \frac{1}{8}S_g$ with the property that the inequality is strict 
at least at one point if $(M,g)$ is conformally diffeomorphic 
to the unit $3$-sphere. By Corollary \ref{Cor2}, assuming $\lambda > 0$ and $\hbox{Osc}(a) < \frac{1}{8}\min_MS_g$, 
phase stability also holds true if $\vert\omega\vert < \omega_a$ is sufficiently large such that $a < \omega^2 + \frac{1}{8}S_g$. As 
a remark, phase stability prevents the existence of arbitrarily large amplitude standing waves (see Corollary \ref{Cor4} in Section \ref{SectBackGrd}).

\medskip As a remark it follows from our proofs that the bounds in Corollary \ref{Cor1} and Corollary \ref{Cor3} are uniform with respect to $\lambda$ 
as long as $\lambda$ stays bounded. In particular, they are uniform with respect to $\lambda$ as $\lambda \to 0$. 

\medskip Let $(S^3,g_0)$ be the unit $3$-sphere. Let $p = 6$. We have that 
$S_{g_0} \equiv 6$. By our theorem we then get that there are $C^{2,\theta}$-bounds for \eqref{SWSyst} 
in the unit 
sphere as soon 
as $a \le 3/4$ and the inequality is strict for at least one point in the manifold. On the other hand, by the noncompactness of the conformal group of $(S^3,g_0)$, 
such bounds do not exist anymore when $a = 3/4$ and $\lambda = 0$. In particular, 
when $a = 3/4$ and $\lambda = 0$, there are sequences of 
solutions $(u_\alpha,\frac{1}{q})$ of \eqref{SWSyst} in $(S^3,g_0)$ with $p = 6$, 
$\alpha \in \mathbb{N}$, which are such that $u_\alpha \rightharpoonup 0$ weakly in $H^1$ but 
not strongly as $\alpha \to +\infty$. Because of 
the noncompactness of the conformal group of $(S^3,g_0)$, the $C^{2,\theta}$-bound for \eqref{SWSyst} when $p = 6$ does not hold true 
in general when we assume the sole \eqref{MainAssumpt}.

\medskip In section \ref{SectBackGrd} we discuss the relation which exists between \eqref{SWSyst}, the Klein-Gordon-Maxwell 
equations, and the Maxwell equations.  A related result to Theorem \ref{MainThm} is presented in Section \ref{SecRes} when we do not 
assume a sign on the scalar curvature of the background metric (implicitly required in Theorem \ref{MainThm} when $p=6$) but ask for the 
potential, balanced by the phase, to be very small (in a non-quantified sense). 
We prove Theorem \ref{MainThm} in sections \ref{Sect1} to \ref{Sect3}. The existence part in the theorem is proved in Section \ref{Sect1}. The 
$C^{2,\theta}$-bound in the subcritical case $p \in (2,6)$ is established in Section \ref{Sect2}. The more delicate $C^{2,\theta}$-bound in the critical 
case $p = 6$ is established in Sections \ref{blowupest} and \ref{Sect3}. Theorem \ref{SecThm} of Section \ref{SecRes} is proved in Section  
\ref{PrfThm2} using the blow-up analysis developed in Sections \ref{blowupest} and \ref{Sect3}.

\section{Action interpretation of the system and the Maxwell equations}\label{SectBackGrd}

We illustrate the background action functional related to our problem and the relation which holds between \eqref{SWSyst}, 
the Maxwell equations, and the Klein-Gordon-Maxwell equations. 
The model we discuss is a model describing the interactions between matter and electromagnetic fields 
established, see, for instance, Benci and Fortunato \cite{BenFor0}, by means of Abelian gauge theories. Formally, the 
ordinary derivatives $\partial_t$ and $\nabla$ in the Klein-Gordon total functional 
are replaced by gauge covariant derivatives given by the rules $\partial_t \to \partial_t + iq\varphi$ and 
$\nabla \to \nabla - iqA$. 
Let $(M,g)$ be a smooth compact Riemannian $3$-manifold, $a >0$ be a smooth positive function in $M$, $\lambda \ge 0$, and $q > 0$. 
We define the Lagrangian densities $\mathcal{L}_0$ and $\mathcal{L}_1$ of $\psi$, $\varphi$, and $A$ by
\begin{equation}\label{EqtLagAct1}
\begin{split}
&\mathcal{L}_0(\psi,\varphi,A) = \frac{1}{2}\left\vert(\frac{\partial}{\partial t} + iq\varphi)\psi\right\vert^2 - \frac{1}{2}\left\vert(\nabla - iqA)\psi\right\vert^2\hskip.1cm\hbox{and}\\
&\mathcal{L}_1(\varphi,A) = \frac{1}{2}\left\vert\frac{\partial A}{\partial t} + \nabla\varphi\right\vert^2 + \frac{\lambda}{2}\vert\varphi\vert^2 - \frac{1}{2}\vert\nabla\times A\vert^2\hskip.1cm ,
\end{split}
\end{equation}
where $\nabla\times$ denotes the curl operator defined thanks to the Hodge dual $\star$ when $M$ is orientable. 
In this model, $\psi$ is a matter field, $(A,\varphi)$ are gauge potentials representing the electromagnetic field $(E,H)$ it generates as in 
\eqref{EqtLagAct6}, $q$ is a nonzero coupling constant, representing the electric charge, and $\lambda \ge 0$ is a coercivity constant which generates phase 
compensation. Then $\mathcal{L}_1$ is a $0$-order perturbation, by $\frac{\lambda}{2}\varphi^2$, 
of the standard electromagnetic Lagrangian density
\begin{eqnarray*}
\mathcal{L}_1^0 &=& 
\frac{1}{2} \left(\left\vert\frac{\partial A}{\partial t} + \nabla\varphi\right\vert^2 - \left\vert\nabla\times A\right\vert^2\right)\\
&=& \frac{1}{2}\left(\vert E\vert^2 - \vert H\vert^2\right)
\end{eqnarray*}
associated to the electromagnetic field $(E,H)$ given by \eqref{EqtLagAct6}. 
We let $\mathcal{S}$ be the total action functional for $\psi$, $\varphi$, $A$ defined by
\begin{equation}\label{EqtLagAct2}
\mathcal{S}(\psi,\varphi,A) = \int\int \left(\mathcal{L}_0 + \mathcal{L}_1 - W\right) dv_gdt\hskip.1cm ,
\end{equation}
where  $\mathcal{L}_0$ and $\mathcal{L}_1$ are as in \eqref{EqtLagAct1}, and $W$ is a function of $\psi$ given by
\begin{equation}\label{EqtLagAct3}
W(\psi) = \frac{a}{2}\vert\psi\vert^2 - \frac{1}{p}\vert\psi\vert^{p}\hskip.1cm .
\end{equation}
Writing $\psi$ in polar form as
$$\psi(x,t) = u(x,t)e^{iS(x,t)}$$
for $u \ge 0$ and $S \in \mathbb{R}\backslash 2\pi\mathbb{Z}$, the total action functional $\mathcal{S}$ given by \eqref{EqtLagAct2} is written as
\begin{equation*}
\begin{split}
\mathcal{S}(u,S,\varphi,A) &= \frac{1}{2}\int\int\left(\left(\frac{\partial u}{\partial t}\right)^2 - \vert\nabla u\vert^2 - au^2\right)dv_gdt  + \frac{1}{p}\int\int u^{p}dv_gdt\\
&+ \frac{1}{2}\int\int \left(\left(\frac{\partial S}{\partial t} + q\varphi\right)^2 - \vert\nabla S-qA\vert^2\right)u^2dv_gdt\\
&+ \frac{1}{2} \int\int \left(\left\vert\frac{\partial A}{\partial t} + \nabla\varphi\right\vert^2 + \lambda\varphi^2 - \left\vert\nabla\times A\right\vert^2\right)dv_gdt\hskip.1cm .
\end{split}
\end{equation*}
Taking the variation of $\mathcal{S}$ with respect to $u$, $S$, $\varphi$, and $A$, we get four equations which are written as
\begin{equation}\label{EqtLagAct5}
\begin{cases}
\frac{\partial^2u}{\partial t^2} + \Delta_gu + au = u^{p-1} + \left(\left(\frac{\partial S}{\partial t} + q\varphi\right)^2 - \vert\nabla S-qA\vert^2\right)u\\
\frac{\partial}{\partial t}\left(\left(\frac{\partial S}{\partial t} + q\varphi\right)u^2\right) - \nabla.\left(\left(\nabla S - qA\right)u^2\right) = 0\\
-\nabla.\left(\frac{\partial A}{\partial t} + \nabla\varphi\right) + \lambda\varphi + q\left(\frac{\partial S}{\partial t} + q\varphi\right)u^2 = 0\\
\nabla\times\left(\nabla\times A\right) + \frac{\partial}{\partial t}\left(\frac{\partial A}{\partial t} + \nabla\varphi\right) = q\left(\nabla S - qA\right)u^2\hskip.1cm .
\end{cases}
\end{equation}
Now let
\begin{equation}\label{EqtLagAct6}
\begin{split}
&E = - \left(\frac{\partial A}{\partial t} + \nabla\varphi\right)\hskip.1cm ,\hskip.1cm H = \nabla\times A\hskip.1cm ,\\
&\rho = - \left(\frac{\partial S}{\partial t} + q\varphi\right)qu^2 - \lambda\varphi
\hskip.1cm ,\hskip.1cm\hbox{and}\hskip.1cm j = \left(\nabla S - qA\right)qu^2\hskip.1cm .
\end{split}
\end{equation}
Then the two last equations in \eqref{EqtLagAct5} give rise to the second couple of the Maxwell equations with respect to a matter distribution 
whose charge and current density are respectively $\rho$ and $j$, namely
\begin{eqnarray}\label{EqtLagAct7}
&&\nabla.E = \rho\label{EqtLagAct7E1}\hskip.1cm\hbox{and}\\
&&\nabla\times H - \frac{\partial E}{\partial t} = j\hskip.1cm ,\label{EqtLagAct7E2}
\end{eqnarray}
while the two first equations in \eqref{EqtLagAct6} give rise to the first couple of the Maxwell equations:
\begin{eqnarray}\label{EqtLagAct8}
&&\nabla\times E + \frac{\partial H}{\partial t} = 0\hskip.1cm\hbox{and}\label{EqtLagAct8E1}\\
&&\nabla. H = 0\hskip.1cm .\label{EqtLagAct8E2}
\end{eqnarray}
In addition, the first equation in \eqref{EqtLagAct5} gives rise to the matter equation
\begin{equation}\label{EqtLagAct9}
\frac{\partial^2u}{\partial t^2} + \Delta_gu + au = u^{p-1} + \frac{\kappa}{q^2u^3}\hskip.1cm ,
\end{equation}
where $\kappa = (\rho + \lambda\varphi)^2-j^2$, while the second equation in \eqref{EqtLagAct5} gives rise to the charge continuity equation
\begin{equation}\label{EqtLagAct10}
\frac{\partial\rho}{\partial t} + \nabla.j = 0
\end{equation}
if we assume that $\varphi = \varphi(x)$. In particular, we recover with \eqref{EqtLagAct5} the Maxwell equations 
\eqref{EqtLagAct7E1}--\eqref{EqtLagAct8E2} together with a Klein-Gordon type equation 
\eqref{EqtLagAct9}. 
The system  \eqref{EqtLagAct5} is referred to as the Klein-Gordon-Maxwell system. 
Suppose now that $S = -\omega t$, $\omega$ real, that $u = u(x)$, and that $A = 0$. Then we are in the electrostatic form of the above equations and we search for 
standing waves for these equations. 
In such a setting, the second and the fourth equations in \eqref{EqtLagAct5} are automatically satisfied, while the first and the third equations in \eqref{EqtLagAct5} become
\begin{equation}\label{EqtLagAct11}
\begin{cases}
\Delta_gu + au = u^{p-1} + (q\varphi-\omega)^2u    \\
\Delta_g\varphi + \lambda\varphi + q(q\varphi-\omega)u^2 = 0\hskip.1cm .
\end{cases}
\end{equation}
In particular, letting $\varphi = \omega v$, we recover our original system \eqref{SWSyst}. 

\medskip A direct consequence of our result concerns the amplitude of standing waves 
for the electrostatic form of the Klein-Gordon-Maxwell system \eqref{EqtLagAct5}. It illustrates the idea that 
phase stability prevents the existence of arbitrarily large amplitude standing waves. 
More precisely, the following corollary is a direct consequence of the results stated in the introduction. 

\begin{cor}[On the amplitude of standing waves]\label{Cor4} Let $(M,g)$ be a smooth compact $3$-dimensional Riemannian manifold and 
$a > 0$ be a smooth positive function in $M$. Let $q > 0$ and $p \in (2,6]$. Assume that one of the three following 
assumptions hold true:\par
(i) $\lambda \ge 0$, $p < 6$, or\par
(ii) $\lambda > 0$, $p=6$, $\hbox{Osc}(a) < \frac{1}{8}\min_MS_g$, or\par
(iii) $\lambda \ge 0$, $p=6$, $a \le \frac{1}{8}S_g$ and the inequality is strict at one point if $(M,g)$\par
is conformally diffeomorphic to the unit $3$-sphere.\\
\noindent Then, for any $\theta \in (0,1)$, there exists $C > 0$ such that $\Vert u\Vert_{C^{2,\theta}} \le C$ for all standing waves $ue^{-i\omega t}$ of \eqref{EqtLagAct5} in its electrostatic form $A = 0$, 
all $\omega \in (-\omega_a,\omega_a)$ in case (i), all $\omega \in K_0(\omega_a-\varepsilon)$ in case (ii), where $\varepsilon > 0$ is suitably chosen, and all 
$\omega \in (-\omega_a,\omega_a)$ in case (iii). Moreover, there also holds that 
$\Vert\varphi\Vert_{C^{2,\theta}} \le C$ for the gauge potential $\varphi$.
\end{cor}

Klein-Gordon-Maxwell systems have been investigated by  
Bechouche, Mauser and Selberg \cite{BecMauSel}, 
Choquet-Bruhat \cite{ChoBru}, 
Deumens \cite{Deu}, 
Eardley and Moncrief \cite{EarMon}, 
Klainerman and Machedon \cite{KlaMac}, 
Machedon and Sterbenz \cite{MacSte}, 
Masmoudi and Nakanishi \cite{MasNak1, MasNak2}, 
Petrescu \cite{Pet}, 
Rodnianski and Tao \cite{RodTao}, and 
Tao \cite{Tao1}.

\medskip Existence of solutions and semiclassical limits for systems like \eqref{SWSyst}, in 
Euclidean space, for subcritical nonlinear terms, have  
been investigated by  
Ambrosetti and Ruiz \cite{AmbRui}, 
D'Aprile and Mugnai \cite{AprMug1,AprMug2}, 
D'Aprile and Wei \cite{AprWei1,AprWei2}, 
D'Avenia and Pisani \cite{AvePis}, 
D'Avenia, Pisani and Siciliano \cite{AvePisSic,AvePisSic2}, 
Azzollini, D'Avenia and Pomponio \cite{AzzAvePom}, 
Azzollini and Pomponio \cite{AzzPom1,AzzPom2}, 
Benci and Fortunato \cite{BenFor0,BenFor2,BenFor3}, 
Bonanno \cite{Bon}, 
Cassani \cite{Cas}, 
Ianni and Vaira \cite{IanVai}, 
Long \cite{Lon}, 
Mugnai \cite{Mug},  
and Ruiz \cite{Rui}.

\medskip Existence and nonexistence of a priori estimates for 
critical elliptic Schr\"odinger equations on manifolds 
have been investigated by Berti-Malchiodi \cite{BerMal}, 
Brendle \cite{Bre,BreSur}, 
Brendle and Marques \cite{BreMar}, 
Druet \cite{Dru1,Dru2}, 
Druet and Hebey \cite{DruHebIMRS, DruHebEinstLich}, 
Druet, Hebey, and V\'etois \cite{DruHebVet}, 
Druet and Laurain \cite{DruLau}, 
Khuri, Marques and Schoen \cite{KhuMarSch}, 
Li and Zhang \cite{LiZha,LiZha2},
Li and Zhu \cite{LiZhu}, 
Marques \cite{Mar}, 
Schoen \cite{Sch3,Sch4}, and 
V\'etois \cite{Vet}. In the subcritical case, a priori estimates for subcritical Schr\"odinger equations goes back   
to the seminal work by Gidas and Spruck \cite{GidSpr}. The above lists are not exhaustive.

\medskip The positive mass theorem in general relativity, that we use below in this paper, was established in Schoen and Yau \cite{SchYau1}. We refer also to 
Schoen and Yau \cite{SchYau2,SchYau3} and Witten \cite{Wit}.

\section{One more estimate}\label{SecRes}

Theorem \ref{MainThm} in the critical case $p = 6$ implicitly requires that the scalar curvature of the background manifold is positive. When this is not the case, 
there are still several situations where \eqref{SWSyst} possesses positive solutions. Such situations are, for instance, easy to obtain by requiring $a$ to be $G$-invariant, where 
$G$ is a subgroup of the isometry group of $G$ having no finite orbits, and by using the improved Sobolev embeddings in Hebey and Vaugon 
\cite{HebVauSym} (see also Hebey \cite{CIMSBook}). Using part of the analysis developed to prove Theorem \ref{MainThm} we may independently get 
a priori bounds for the set of solutions when the mass potential, balanced by the phase, is sufficiently small (in a non-quantified sense). In particular, we can prove that 
the following theorem holds true. Given $R > 0$ we let $B^{0,1}_R$ be the set of smooth functions $a \in C^\infty(M)$ such 
that $\Vert a\Vert_{C^{0,1}} \le R$. 

\begin{thm}\label{SecThm} Let $(M,g)$ be a smooth compact $3$-dimensional Riemannian manifold, $\lambda \ge 0$, $R > 0$, $q > 0$, and $\theta \in (0,1)$. 
 There exist $\varepsilon, C > 0$, depending only on $M$, $g$, $R$, $\lambda$, $q$, and $\theta$, such that for any $a \in B^{0,1}_R$, 
 $a > 0$, and any $\omega \in \left(-\omega_a,\omega_a\right)$, if
 
 (i) $a-\omega^2 < \varepsilon$ and $\lambda > 0$\hskip.1cm ,\hskip.1cm or
 
 (ii) $a < \varepsilon$ and $\lambda \ge 0$\hskip.1cm ,
 
 \noindent then $\Vert\mathcal{U}\Vert_{C^{2,\theta}} \le C$ for all $\mathcal{U} \in \mathcal{S}_6(\omega^\prime)$, all 
 $\omega^\prime \in \left(-\omega_a,-\omega\right]\bigcup\left[\omega,\omega_a\right)$ in case (i), and all 
 $\omega^\prime \in \left(-\omega_a,\omega_a\right)$ in case (ii).
\end{thm}

A similar corollary to Corollary 1.1 can be derived from the above theorem. In particular, we prevent the existence of arbitrarily large 
amplitude standing waves for the Klein-Gordon-Maxwell system in electrostatic form when either (i) or (ii) is assumed to hold. The $C^{0,1}$-bound 
on $a$ in Theorem \ref{SecThm} can be lowered. We require the $C^{0,1}$-bound to get $C^{2,\theta}$-estimates on solutions without restrictions on $\theta$. 
Theorem \ref{SecThm} is proved in Section \ref{PrfThm2}.

\section{Variational analysis and the existence part of Theorem \ref{MainThm}}\label{Sect1}

Solutions $(u,v)$ of \eqref{SWSyst} are critical points of a functional $S$ defined on $H^1\times H^1$ where the $H^1$-norms of $u$ and $v$ compete 
one with another. More precisely, 
\begin{equation}\label{FullFunct}
\begin{split}
S(u,v) & = \frac{1}{2}\int_M\vert\nabla u\vert^2dv_g - \frac{\omega^2}{2}\int_M\vert\nabla v\vert^2dv_g + \frac{1}{2}\int_Mau^2dv_g\\
&-\frac{\omega^2\lambda}{2}\int_Mv^2dv_g- \frac{1}{p}\int_Mu^pdv_g - \frac{\omega^2}{2}\int_Mu^2(1-qv)^2dv_g\hskip.1cm .
\end{split}
\end{equation}
As is easily checked, $S(0,v) < 0$ if $v$ is nonconstant and, for any $R > 0$, there exists $u \in H^1$ such that 
$\Vert u\Vert_{H^1} \ge R$ and $S(tu,0) > 0$ for all $t \in (0,1]$. 
In order to overcome the problems caused by  the competition between $u$ and $v$ in $S$, following 
the very nice idea in Benci and Fortunato \cite{BenFor0}, we introduce the map 
$\Phi: H^1 \to H^1$ defined by the equation
\begin{equation}\label{DefPhi}
\Delta_g\Phi(u) + \left(\lambda + q^2u^2\right)\Phi(u) = qu^2\hskip.1cm .
\end{equation}
It follows from standard variational arguments that $\Phi$ is well-defined in $H^1$ as soon as $\lambda > 0$. Noting that 
$u^2\Phi(u) \in L^2$ since $u$ and $\Phi(u)$ are in $H^1$, it follows from \eqref{DefPhi} that $\Phi(u) \in H^2$. We further get 
with \eqref{DefPhi} that $\Phi(u) \in H^{2,3}$.  
Elementary though useful properties of $\Phi$ are as follows.

\begin{lem}\label{C1Phi} The map $\Phi: H^1 \to H^1$ is $C^1$ and its differential $D\Phi(u) = V_u$ at $u$ is the map defined by
\begin{equation}\label{EqtC1Phi}
\Delta_gV_u(h) + \left(\lambda + q^2u^2\right)V_u(h) = 2qu\left(1-q\Phi(u)\right)h
\end{equation}
for all $h \in H^1$. Moreover, it holds that
\begin{equation}\label{EqtDbleIneq}
0 \le \Phi(u) \le \frac{1}{q}
\end{equation}
for all $u \in H^1$.
\end{lem}

\begin{proof}[Proof of Lemma \ref{C1Phi}] Let $u \in H^1$. It is clear from the maximum principle that $\Phi(u) \ge 0$. Noting that 
\begin{equation}\label{Eqt0LemC1}
\Delta_g\left(\frac{1}{q}-\Phi(u)\right) + \left(\lambda + q^2u^2\right)\left(\frac{1}{q}-\Phi(u)\right) = \frac{\lambda}{q}
\end{equation}
it also follows from the maximum principle that $\Phi(u) \le \frac{1}{q}$, and this proves \eqref{EqtDbleIneq}. Given $h \in H^1$, we compute
\begin{equation}\label{Lem1Eqt1}
\begin{split}
&\Delta_g\left(\Phi(u+h)-\Phi(u) \right) + \left(\lambda + q^2u^2\right)\left(\Phi(u+h)-\Phi(u)\right)\\
&= q\left(1-q\Phi(u+h)\right)\left(h^2+2uh\right)\hskip.1cm .
\end{split}
\end{equation}
Multiplying \eqref{Lem1Eqt1} by $\Phi(u+h)-\Phi(u)$, and integrating over $M$, by the coercivity of $\Delta_g + \lambda$, by the Sobolev embedding theorem, 
by H\"older's inequality, and by \eqref{EqtDbleIneq}, we get that
$$\left\Vert\Phi(u+h)-\Phi(u)\right\Vert_{H^1} \le C(u)\Vert h\Vert_{H^1}\left(1+\Vert h\Vert_{H^1}\right)$$
for all $h \in H^1$, where $C(u) > 0$ is independent of $h$. In particular, $\Phi$ is continuous. Also we compute, 
\begin{equation}\label{Lem1Eqt2}
\begin{split}
&\Delta_g\left(\Phi(u+h)-\Phi(u) - V_u(h)\right) + \left(\lambda + q^2u^2\right)\left(\Phi(u+h)-\Phi(u) - V_u(h)\right)\\
&= qh^2 - q^2\Phi(u+h)h^2 + 2q^2u\left(\Phi(u)-\Phi(u+h)\right)h\hskip.1cm .
\end{split}
\end{equation}
Multiplying \eqref{Lem1Eqt2} by $\Phi(u+h)-\Phi(u) - V_u(h)$ and integrating over $M$, by the coercivity of $\Delta_g + \lambda$, by the Sobolev embedding theorem, 
by H\"older's inequality, and by \eqref{EqtDbleIneq}, we get that
$$\left\Vert\Phi(u+h)-\Phi(u)-V_u(h)\right\Vert_{H^1} \le C(u)\left\Vert h\right\Vert_{H^1}\left(\left\Vert h\right\Vert_{H^1}+\Vert\Phi(u+h)-\Phi(u)\Vert_{H^1}\right)$$
for all $h \in H^1$, where $C(u) > 0$ is independent of $h$. Since $\Phi$ is continuous, the differentiability of $\Phi$ at $u$, as well as the fact 
that $D\Phi(u) = V_u$, follow from this estimate. The continuity of $u \to V_u$ can be proved with similar arguments. This proves the lemma.
\end{proof}

Coming back to \eqref{Lem1Eqt1}, and the procedure described after \eqref{Lem1Eqt1}, it holds true that
\begin{equation}\label{LipsType}
\Vert\Phi(v)-\Phi(u)\Vert_{H^1} \le C(u,v)\Vert v-u\Vert_{L^3}
\end{equation}
for all $u, v \in H^1$, where $C(u,v) = C\bigl(\Vert u\Vert_{L^3} + \Vert v\Vert_{L^3}\bigr)$ and $C > 0$ 
is independent of $u$ and $v$. We can also write that
\begin{equation}\label{LipsType2}
\Vert V_v(h)-V_u(h)\Vert_{H^1} \le C(u,v)\Vert v-u\Vert_{L^3}\Vert h\Vert_{L^6}
\end{equation}
for all $u, v, h \in H^1$, where $C(u,v)= C\bigl(1+\Vert u\Vert_{L^3}^2 + \Vert v\Vert_{L^3}^2\bigr)$ and $C > 0$ is independent of $u$ and $v$.

\begin{lem}\label{RestDiff} The map $\Theta: H^1 \to \mathbb{R}$ given by 
\begin{equation}\label{DefTheta}
\Theta(u) = \frac{1}{2}\int_M\left(1-q\Phi(u)\right)u^2dv_g
\end{equation}
is $C^1$ and, for any $u \in H^1$,
\begin{equation}\label{DiffTheta}
D\Theta(u).(h) = \int_M\left(1-q\Phi(u)\right)^2uhdv_g
\end{equation}
for all $h \in H^1$.
\end{lem}

\begin{proof}[Proof of Lemma \ref{RestDiff}] It follows from Lemma \ref{C1Phi} that $\Theta$ is $C^1$. By \eqref{DefPhi},
$$\Theta(u) = \frac{1}{2}\int_M\left(\vert\nabla\Phi(u)\vert^2+\lambda\Phi(u)^2\right)dv_g 
+ \frac{1}{2}\int_M\left(1-q\Phi(u)\right)^2u^2dv_g\hskip.1cm ,$$
and we also have that $\frac{\partial H}{\partial\Phi}\left(u,\Phi(u)\right) = 0$, where 
$$H(u,\Phi) = \frac{1}{2}\int_M\left(\vert\nabla\Phi\vert^2+\lambda\Phi^2\right)dv_g + \frac{q^2}{2}\int_Mu^2\Phi^2dv_g 
- q\int_Mu^2\Phi dv_g\hskip.1cm .$$
Noting that
$$\Theta(u) = H\left(u,\Phi(u)\right) + \frac{1}{2}\int_Mu^2dv_g\hskip.1cm ,$$
we get that \eqref{DiffTheta} holds true. This ends the proof of the lemma.
\end{proof}

At this point we define the functional $I_p: H^1 \to \mathbb{R}$ by
\begin{equation}\label{DefFct}
\begin{split}
I_p(u) &= \frac{1}{2}\int_M\vert\nabla u\vert^2dv_g + \frac{1}{2}\int_Mau^2dv_g - \frac{1}{p}\int_M(u^+)^pdv_g\\
&- \frac{\omega^2}{2}\int_M\left(1-q\Phi(u)\right)(u^+)^2dv_g\hskip.1cm ,
\end{split}
\end{equation}
where $u^+ = \max(u,0)$, and $p \in (2,6]$. By Lemma \ref{C1Phi} we get that $I_p$ is $C^1$ and if $\hat\Theta$ is given by
\begin{equation}\label{DefHatTheta}
\hat\Theta(u) = \int_M\left(1-q\Phi(u)\right)(u^+)^2dv_g
\end{equation}
then, for any $u \in H^1$,
\begin{equation}\label{DiffHatTheta}
D\hat\Theta(u).(h) = 2\int_M\left(1-q\Phi(u)\right)u^+hdv_g - q\int_M(u^+)^2V_u(h)dv_g
\end{equation}
for all $h \in H^1$, where $V_u$ is as in \eqref{EqtC1Phi}. First we prove the existence part of 
Theorem \ref{MainThm} when $p \in (2,6)$. 
For this aim we use the mountain pass lemma, as stated in Ambrosetti and Rabinowitz 
\cite{AmbRab}, that we apply to the functional $I_p$ defined in \eqref{DefFct}.

\begin{proof}[Proof of the Existence Part in Theorem \ref{MainThm} when $p \in (2,6)$] Suppose first that $\lambda = 0$. Since $p < 6$ 
and $a > 0$ there exists $u$ 
smooth and positive such that
$$\Delta_gu + au = u^{p-1}\hskip.1cm .$$
The existence of $u$ easily follows from standard variational arguments. 
Then $(u,\frac{1}{q})$ solve \eqref{SWSyst}. From now on we assume $\lambda > 0$. 
It is easily checked that $I_p(0) = 0$ and that for $u_0$ 
arbitrarily given in $H^1$, with $u_0^+ \not\equiv 0$, there holds that $I_p(tu_0) \to -\infty$ as $t \to +\infty$. Moreover, since we assumed that $\omega \in (-\omega_a,\omega_a)$, we get 
from \eqref{EqtDbleIneq} and the Sobolev embedding theorem that 
\begin{equation}\label{Eqt1PExist1}
\begin{split}
I_p(u) &\ge \frac{1}{2}\int_M\vert\nabla u\vert^2dv_g + \frac{1}{2}\int_M\left(a-\omega^2\right)u^2dv_g - \frac{1}{p}\int_M\vert u\vert^pdv_g\\
&\ge C_1\Vert u\Vert_{H^1}^2 - C_2\Vert u\Vert_{H^1}^p\hskip.1cm ,
\end{split}
\end{equation}
where $C_1, C_2 > 0$. In particular, it follows from \eqref{Eqt1PExist1} that there exist $\delta > 0$ and $C > 0$ such that 
$I_p(u) \ge C$ for all $u \in H^1$ which satisfy $\Vert u\Vert_{H^1} = \delta$. Let $T_0 = T(u_0)$ be such that $I_p(T_0u_0) < 0$. 
Let $c_p = c_p(u_0)$ be given by
\begin{equation}\label{Eqt2PExist1}
c_p = \inf_{P\in\mathcal{P}}\max_{u\in P}I_p(u)\hskip.1cm ,
\end{equation}
where $\mathcal{P}$ denotes the class of continuous paths joining $0$ to $T_0u_0$. 
By the mountain pass lemma, see Ambrosetti and Rabinowitz 
\cite{AmbRab}, there exists $(u_\alpha)_\alpha$ in $H^1$ such that $I_p(u_\alpha) \to c_p$ 
and $DI_p(u_\alpha) \to 0$ as $\alpha \to +\infty$. Writing that $I_p(u_\alpha) = c_p + o(1)$ 
and that $DI_p(u_\alpha).(u_\alpha) = o(\Vert u_\alpha\Vert_{H^1})$ we get that
\begin{equation}\label{Eqt3PExist1}
\begin{split}
&\frac{1}{2}\int_M\left(\vert\nabla u_\alpha\vert^2+ au_\alpha^2\right)dv_g\\
&\hskip.4cm= \frac{1}{p}\int_M(u_\alpha^+)^pdv_g + c_p + \frac{\omega^2}{2}\hat\Theta(u_\alpha) + o(1)
\hskip.1cm ,\hskip.1cm\hbox{and}\\
&\frac{1}{2}\int_M\left(\vert\nabla u_\alpha\vert^2+ au_\alpha^2\right)dv_g\\
&\hskip.4cm= \frac{1}{2}\int_M(u_\alpha^+)^pdv_g + \frac{\omega^2}{4}D\hat\Theta(u_\alpha).(u_\alpha) + o(\Vert u_\alpha\Vert_{H^1})\hskip.1cm ,
\end{split}
\end{equation}
where $\hat\Theta$ is as in \eqref{DefHatTheta}. By \eqref{EqtC1Phi} and \eqref{DiffHatTheta}, for any $u \in H^1$,
$$\Delta_g\left(V_u(u) - 2\Phi(u)\right) + \left(\lambda + q^2u^2\right)\left(V_u(u)-2\Phi(u)\right) 
= -2q^2u^2\Phi(u) \le 0\hskip.1cm .$$
By the maximum principle and \eqref{EqtDbleIneq} we then get that 
\begin{equation}\label{Eqt4PExist1}
0 \le V_u(u) \le 2\Phi(u) \le \frac{2}{q}
\end{equation}
for all $u \in H^1$. In particular, thanks to \eqref{DiffHatTheta} and \eqref{Eqt4PExist1}, we get that
\begin{equation}\label{Eqt5PExist1}
\left\vert D\hat\Theta(u).(u)\right\vert \le C\int_M(u^+)^2dv_g
\end{equation}
for all $u \in H^1$, where $C > 0$ is independent of $u$. 
Substracting the second equation in \eqref{Eqt3PExist1} to the first, thanks to \eqref{Eqt4PExist1}, 
it follows that
\begin{equation}\label{Eqt6PExist1}
\begin{split}
\left(\frac{1}{2}-\frac{1}{p}\right)\Vert u_\alpha^+\Vert_{L^p}^p 
&\le c_p + o(1) + C\Vert u_\alpha^+\Vert_{L^2}^2+o(\Vert u_\alpha\Vert_{H^1})\\
&\le c_p + o(1) + C^\prime\Vert u_\alpha^+\Vert_{L^p}^2+o(\Vert u_\alpha\Vert_{H^1})
\end{split}
\end{equation}
for all $\alpha$, where $C, C^\prime > 0$ are independent of $\alpha$. By \eqref{Eqt3PExist1} and \eqref{Eqt6PExist1}, since $p > 2$, 
the sequence $(u_\alpha)_\alpha$ is bounded in $H^1$. Since $p < 6$ we may then assume that 
there exists $u_p \in H^1$ such that, up to a subsequence, 
$u_\alpha \rightharpoonup u_p$ in $H^1$ and $u_\alpha \to u_p$ in $L^p\cap L^3$ as $\alpha \to +\infty$. 
Since $c_p > 0$, it is clear from \eqref{Eqt3PExist1} that $u_p\not\equiv 0$.
For any $\varphi \in H^1$, 
$DI_p(u_\alpha).(\varphi) = o(1)$. Letting $\alpha \to +\infty$ in this equation, thanks to \eqref{LipsType}, \eqref{LipsType2},  
and \eqref{DiffTheta}, it follows that for any $\varphi \in H^1$,
\begin{equation}\label{Eqt7PExist1}
\begin{split}
&\int_M(\nabla u_p\nabla\varphi)dv_g + \int_Mau_p\varphi dv_g\\
&= \int_M(u_p^+)^{p-1}\varphi dv_g + \omega^2\int_M\left(1-q\Phi(u_p)\right)u_p^+\varphi dv_g\\
&\hskip.4cm- \frac{q\omega^2}{2}\int_MV_{u_p}(\varphi)(u_p^+)^2dv_g\hskip.1cm .
\end{split}
\end{equation}
Noting that $V_{u_p}(u_p^-) \le 0$, where $u^- = \max(-u,0)$, we get from \eqref{Eqt7PExist1} that
$$\int_M\left((\nabla u_p\nabla u_p^-) +au_pu_p^-\right)dv_g \ge 0$$
and it follows that $u_p^- \equiv 0$. Hence $u_p \ge 0$ in $M$. By Lemma \ref{RestDiff} and \eqref{Eqt7PExist1} 
we can then write that for any $\varphi \in H^1$,
\begin{equation}\label{Eqt8PExist1}
\begin{split}
&\int_M(\nabla u_p\nabla\varphi)dv_g + \int_Mau_p\varphi dv_g\\
&= \int_Mu_p^{p-1}\varphi dv_g + \omega^2\int_M\left(1-q\Phi(u_p)\right)^2u_p\varphi dv_g\hskip.1cm .
\end{split}
\end{equation}
In particular, by \eqref{Eqt8PExist1}, we get that $\left(u_p,\Phi(u_p)\right)$ solves \eqref{SWSyst}. By the maximum principle and elliptic regularity 
we get that $u_p > 0$, $\Phi(u_p) >0$, and $u_p$ and $\Phi(u_p)$ are smooth. This ends the proof of 
the existence part in Theorem \ref{MainThm}  when $p \in (2,6)$.
\end{proof}

As already mentionned, it follows from \eqref{Eqt3PExist1} that the sequence $(u_\alpha)_\alpha$ is bounded in $H^1$. By 
\eqref{LipsType} and \eqref{LipsType2} we then get that
$$D\hat\Theta(u_\alpha).(u_\alpha) \to D\hat\Theta(u_p).(u_p)$$
as $\alpha \to +\infty$ since $u_\alpha \to u_p$ in $L^p\cap L^3$. Also there holds that $\int_Mau_\alpha^2dv_g \to \int_Mau_p^2dv_g$ and 
$\int_M(u_\alpha^+)^pdv_g \to \int_Mu_p^pdv_g$ as $\alpha \to +\infty$. By \eqref{Eqt7PExist1},
$$\frac{1}{2}\int_M\left(\vert\nabla u_p\vert^2+ au_p^2\right)dv_g = \frac{1}{2}\int_Mu_p^pdv_g + \frac{\omega^2}{4}D\hat\Theta(u_p).(u_p)\hskip.1cm .$$
Coming back to the second equation in \eqref{Eqt3PExist1} we get that
$$\int_M\vert\nabla u_\alpha\vert^2dv_g \to \int_M\vert\nabla u_p\vert^2dv_g$$
as $\alpha \to +\infty$ and it follows that $u_\alpha \to u_p$ in $H^1$ as $\alpha \to +\infty$. By the first equation in \eqref{Eqt3PExist1}, since by 
\eqref{LipsType} we can write that $\hat\Theta(u_\alpha) \to \hat\Theta(u_p)$ as $\alpha \to +\infty$, we get that
\begin{equation}\label{Levelcp}
I_p(u_p) = c_p\hskip.1cm ,
\end{equation}
where $c_p$ is as in \eqref{Eqt2PExist1}.

\medskip At this point it remains to prove the existence part of Theorem \ref{MainThm} when $p = 6$. For this aim we use the existence of subcritical solutions 
we just obtained and the developments in Schoen \cite{Sch1} based on the positive mass theorem of Schoen and Yau \cite{SchYau1} (see also 
Schoen and Yau \cite{SchYau2,SchYau3} and Witten \cite{Wit}). 
Let $G$ be the Green's function of the conformal Laplacian $\Delta_g + \frac{1}{8}S_g$. Let $x_0$ be given in $M$ and 
$G(x) = G(x_0,x)$. In geodesic normal 
coordinates,
\begin{equation}\label{GreenSchYau}
G(x) = \frac{1}{\omega_2\vert x\vert} + A + \alpha(x)\hskip.1cm ,
\end{equation}
where $\omega_2$ is the volume of the unit $2$-sphere, and $\alpha(x) = O(\vert x\vert)$. 
Noting that $\hat g = G^4g$ is scalar flat and asymptotically Euclidean, 
it is a consequence of the positive mass theorem that $A \ge 0$ and $A = 0$ if and only if $(M,g)$ is 
conformally diffeomorphic to the unit sphere. Following Schoen \cite{Sch1}, we let $\rho_0 > 0$ be a small radius 
and $\varepsilon_0 > 0$ to be chosen small relative to $\rho_0$. Let also $\psi$ be a piecewise smooth decreasing function 
of $\vert x\vert$ such that $\psi(x) = 1$ for $\vert x\vert \le \rho_0$, $\psi(x) = 0$ for $\vert x\vert \ge 2\rho_0$, and 
$\vert\nabla\psi\vert \le \rho_0^{-1}$ for $\rho_0 \le \vert x\vert \le 2\rho_0$. We define $u_\varepsilon$, $\varepsilon > 0$, by
\begin{equation}\label{DefuEps}
\begin{cases}
u_\varepsilon(x) = \left(\frac{\varepsilon}{\varepsilon^2+d_g(x_0,x)^2}\right)^{1/2}\hskip.2cm\hbox{for}~d_g(x_0,x) \le \rho_0\hskip.1cm ,\\
u_\varepsilon(x) = \varepsilon_0\left(G(x) - \psi(x)\alpha(x)\right)\hskip.2cm\hbox{for}~\rho_0 \le d_g(x_0,x) \le 2\rho_0\hskip.1cm ,\\
u_\varepsilon(x) = \varepsilon_0 G(x)\hskip.2cm\hbox{for}~d_g(x_0,x) \ge 2\rho_0\hskip.1cm .
\end{cases}
\end{equation}
and require that
$$\varepsilon_0\left(\frac{1}{\omega_2\rho_0}+A\right) = \sqrt{\frac{\varepsilon}{\varepsilon^2+\rho_0^2}}\hskip.1cm .$$
Then, see Schoen \cite{Sch1}, since $A > 0$ if $(M,g)$ is not conformally diffeomorphic to $(S^3,g_0)$, we get that 
\begin{equation}\label{QuotSch}
\frac{\int_M\left(\vert\nabla u_\varepsilon\vert^2+\frac{1}{8}S_gu_\varepsilon^2\right)dv_g}{\left(\int_Mu_\varepsilon^6dv_g\right)^{1/3}} 
< \frac{1}{K_3^2}
\end{equation}
for $\varepsilon \ll 1$, when $(M,g)$ is not conformally diffeomorphic to $(S^3,g_0)$, where $K_3$ is the sharp constant in the Euclidean 
Sobolev inequality
\begin{equation}\label{EuclIneq}
\left(\int_{\mathbb{R}^3}\vert u\vert^6dx\right)^{1/3} \le K_3^2\int_{\mathbb{R}^3}\vert\nabla u\vert^2dx\hskip.1cm .
\end{equation}
Also there holds that
\begin{equation}\label{QuotSch2}
\int_Mu_\varepsilon^6dv_g = \int_{\mathbb{R}^3}\left(\frac{1}{1+\vert x\vert^2}\right)^3dx + o(1)\hskip.1cm .
\end{equation}
Now we split the proof of the existence part of Theorem \ref{MainThm} when $p = 6$ in two cases. In the first case we assume that 
$(M,g)$ is not conformally diffeomorphic to the unit $3$-sphere. In the second case we assume that 
$(M,g)$ is conformally diffeomorphic to the unit $3$-sphere. 

\medskip We use in what follows phase compensation for 
electrostatic Klein-Gordon-Maxwell systems. 
Phase compensation follows from the subcritical nature of the second equation in \eqref{SWSyst}.
It is a key tool in studying \eqref{SWSyst} and it 
can be explained in naive terms in the following way: if $(u,v)$ solves \eqref{SWSyst}, and 
$u$ is small in $L^{p^\prime}$, for $p^\prime$ sufficiently large but still in the subcritical range, then, by Sobolev embeddings, $v$ needs to be small in $L^\infty$, and the 
potential like term 
$a-\omega^2(qv-1)^2$ in the nonlinear Schr\"odinger equation in \eqref{SWSyst} approaches $a-\omega^2$. In particular, the nonlinear 
Schr\"odinger equation in \eqref{SWSyst}, and its variational formulations, approach the static isolated nonlinear model Schr\"odinger equation 
with unknown function $u$, potential $a-\omega^2$, and nonlinear term $u^{p-1}$.

\begin{proof}[Proof of the Existence Part in Theorem \ref{MainThm} when $p = 6$. Case 1] We assume 
that $(M,g)$ is not conformally diffeomorphic to the unit $3$-sphere. 
Suppose first that $\lambda = 0$ or that $\omega = 0$. By \eqref{MainAssumpt} and \eqref{QuotSch} we then get that 
$a \le \frac{1}{8}S_g$ and that 
$$\frac{\int_M\left(\vert\nabla u_\varepsilon\vert^2+au_\varepsilon^2\right)dv_g}{\left(\int_Mu_\varepsilon^6dv_g\right)^{1/3}} 
< \frac{1}{K_3^2}$$
for $\varepsilon \ll 1$. In particular, see for instance Aubin \cite{Aub1,Aub2}, 
there exists $u$ smooth and positive such that
$$\Delta_gu+au = u^5\hskip.1cm .$$
Then $(u,\frac{1}{q})$ solves \eqref{SWSyst} if $\lambda = 0$, and $\left(u,\Phi(u)\right)$ solves \eqref{SWSyst} if 
$\lambda > 0$. From now on we assume that $\lambda > 0$ and $\omega^2 > 0$. 
Let $(\varepsilon_\alpha)_\alpha$ be a sequence of positive real numbers such that $\varepsilon_\alpha \to 0$ 
as $\alpha \to +\infty$. Let $u_\alpha = u_{\varepsilon_\alpha}$. By \eqref{QuotSch2}, there exists $T > 0$, independent 
of $\alpha$, such that $I_6(Tu_\alpha) < 0$ for all $\alpha \gg 1$. By \eqref{QuotSch2}, noting that $u_\alpha \to 0$ a.e. 
as $\alpha \to +\infty$, we get that $u_\alpha \to 0$ in $L^q$ for all $q < 6$ as $\alpha \to +\infty$. Let $(t_\alpha)_\alpha$ 
be any sequence in $[0,T]$. We have that 
$\Phi(0) = 0$ and thus, by \eqref{LipsType}, there holds that $\Phi(t_\alpha u_\alpha) \to 0$ 
in $H^1$ as $\alpha \to +\infty$. By \eqref{DefPhi} and \eqref{EqtDbleIneq} we then get that 
$\Phi(t_\alpha u_\alpha) \to 0$ in $H^{2,q}$ for all $q < 3$ as $\alpha \to +\infty$. In particular, $\Phi(t_\alpha u_\alpha) \to 0$ 
in $L^\infty$ as $\alpha \to +\infty$, and we can write that
\begin{equation}\label{Conv0ArgEx1}
\max_{0\le t \le T}\Vert\Phi(tu_\alpha)\Vert_{L^\infty} \to 0
\end{equation}
as $\alpha \to +\infty$. Since $k\lambda < 1$, we get with \eqref{Conv0ArgEx1} that for any $\alpha \gg 1$, and any $t \in [0,T]$,
\begin{equation}\label{IneqC0Ex1}
\int_M\left(1-q\Phi(t u_\alpha)\right)u_\alpha^2dv_g \ge k\lambda \int_Mu_\alpha^2dv_g
\hskip.1cm
\end{equation}
Let $\mathcal{F}_6$ be the functional defined in $H^1$ by
\begin{equation}\label{DefF6Proof1}
\mathcal{F}_6(u) = \frac{1}{2}\int_M\vert\nabla u\vert^2dv_g + \frac{1}{16}\int_MS_gu^2dv_g - 
\frac{1}{6}\int_M\vert u\vert^6dv_g\hskip.1cm .
\end{equation}
By \eqref{MainAssumpt} and \eqref{IneqC0Ex1},
\begin{equation}\label{CompFctsProof1Eqt1}
\max_{0\le t\le T}I_6(tu_\alpha) \le \max_{0\le t\le T}\mathcal{F}_6(tu_\alpha)
\end{equation}
for all $\alpha \gg 1$. Fix $u_0 = u_\alpha$ for $\alpha \gg 1$, sufficiently large such that \eqref{QuotSch} holds true, 
and let $T_0 = T$. For $\varepsilon > 0$ sufficiently small, $I_p(T_0u_0) < 0$ and
\begin{equation}\label{CompFctsProof1Eqt2}
\max_{0\le t\le T_0}I_p(tu_0) \le (1+\delta_\varepsilon)\max_{0\le t\le T_0}I_6(tu_0)
\end{equation}
for all $p \in (6-\varepsilon,6)$, where $\delta_\varepsilon >0$ is such that $\delta_\varepsilon \to 0$ 
as $\varepsilon \to 0$. As is easily checked, differentiating $\mathcal{F}_6(tu_0)$ 
with respect to $t$, we get that
\begin{equation}\label{MaxF6Proof1}
\max_{0 \le t \le T_0}\mathcal{F}_6(tu_0) \le \left(\frac{1}{2}-\frac{1}{6}\right) 
\left(\frac{\int_M\left(\vert\nabla u_0\vert^2+\frac{1}{8}S_gu_0^2\right)dv_g}{\left(\int_Mu_0^6dv_g\right)^{1/3}}\right)^{3/2}\hskip.1cm .
\end{equation}
We have that
$$c_p \le \max_{0\le t\le T_0}I_p(tu_0)\hskip.1cm ,$$
where $c_p$ is as in \eqref{Eqt2PExist1}. 
By \eqref{QuotSch}, \eqref{CompFctsProof1Eqt1}, \eqref{CompFctsProof1Eqt2}, and \eqref{MaxF6Proof1}, we then get 
that there exists $\delta_0 > 0$ such that for $0 < \varepsilon \ll 1$ sufficiently small, 
\begin{equation}\label{Proof2ExEqt3}
\delta_0 \le c_p \le \frac{1}{3K_3^3} - \delta_0
\end{equation}
for all $p \in (6-\varepsilon,6)$. Let $u_p$, $p < 6$, be the solution obtained in the subcritical part of the proof 
of existence. We have, see \eqref{Levelcp}, that $I_p(u_p) = c_p$ and also that 
$I_p^\prime(u_p) = 0$. In particular,
\begin{equation}\label{EnerEqtUp1}
\frac{1}{2}\int_M\left(\vert\nabla u_p\vert^2 + au_p^2\right)dv_g
= c_p + \frac{1}{p}\int_Mu_p^pdv_g + \frac{\omega^2}{2}\int_M\left(1-q\Phi(u_p)\right)u_p^2dv_g\hskip.1cm ,
\end{equation}
and
\begin{equation}\label{EnerEqtUp2}
\int_M\left(\vert\nabla u_p\vert^2 + au_p^2\right)dv_g = \int_Mu_p^pdv_g 
+ \omega^2\int_M\left(1-q\Phi(u_p)\right)^2u_p^2dv_g\hskip.1cm .
\end{equation}
Let $(p_\alpha)_\alpha$ be a sequence such that $p_\alpha < 6$ for all $\alpha$ and $p_\alpha\to 6$ as $\alpha \to +\infty$. 
Let $v_\alpha = u_{p_\alpha}$. Substracting \eqref{EnerEqtUp1}-$\frac{1}{2}$\eqref{EnerEqtUp2}, we get 
from \eqref{EqtDbleIneq} and \eqref{Proof2ExEqt3} that the sequence consisting of the $\Vert v_\alpha\Vert_{L^{p_\alpha}}$'s is bounded. Then, 
by \eqref{EnerEqtUp1}, $(v_\alpha)_\alpha$ is bounded in $H^1$. 
In particular, there exists $u \in H^1$, $u \ge 0$, such that $v_\alpha \rightharpoonup u$ weakly in $H^1$, 
$v_\alpha \to u$ strongly in $L^3$, and $v_\alpha \to u$ a.e. as $\alpha \to +\infty$. Since $I_{p_\alpha}^\prime(v_\alpha) = 0$, it follows from 
\eqref{LipsType}--\eqref{LipsType2} that there also holds that $I_6^\prime(u) = 0$. By the Trudinger \cite{Tru} regularity argument 
developed for critical Schr\"odinger equations, by standard 
elliptic regularity, 
and by the maximum principle, either $u \equiv 0$, or $u > 0$ in $M$ and $u, \Phi(u)$ are smooth. Thus it remains only to prove that 
$u\not\equiv 0$. By the sharp Sobolev inequality, as established in Hebey and Vaugon \cite{HebVau1,HebVau2}, there exists $B >0$ such  
that
\begin{equation}\label{ConclEqtExistThm1Eqt1}
\left(\int_Mv_\alpha^{p_\alpha}dv_g\right)^{2/p_\alpha} \le \left(K_3^2+o(1)\right)\int_M\left(\vert\nabla v_\alpha\vert^2+av_\alpha^2\right)dv_g 
+ B \int_Mv_\alpha^2dv_g
\end{equation}
for all $\alpha$. By contradiction we assume $u\equiv 0$. Substracting \eqref{EnerEqtUp1}-$\frac{1}{p_\alpha}$\eqref{EnerEqtUp2} 
we get that
\begin{equation}\label{ConclEqtExistThm1Eqt2}
\int_M\left(\vert\nabla v_\alpha\vert^2+av_\alpha^2\right)dv_g = \frac{2p_\alpha}{p_\alpha-2}c_{p_\alpha} + o(1)\hskip.1cm ,
\end{equation}
and we also have that
\begin{equation}\label{ConclEqtExistThm1Eqt3}
\int_M\left(\vert\nabla v_\alpha\vert^2+av_\alpha^2\right)dv_g = 
\int_Mv_\alpha^{p_\alpha}dv_g + o(1)\hskip.1cm .
\end{equation}
Inserting \eqref{ConclEqtExistThm1Eqt2}--\eqref{ConclEqtExistThm1Eqt3} into \eqref{ConclEqtExistThm1Eqt1} we obtain that
\begin{equation}\label{LastEqtExistProof}
\left(\frac{2p_\alpha}{p_\alpha-2}c_{p_\alpha} + o(1)\right)^{2/p_\alpha} \le 
K_3^2\left(\frac{2p_\alpha}{p_\alpha-2}c_{p_\alpha} + o(1)\right)
\end{equation}
for all $\alpha$, and the contradiction follows from \eqref{Proof2ExEqt3} by letting $\alpha \to +\infty$ in \eqref{LastEqtExistProof}. This ends the 
proof of the existence part of Theorem \ref{MainThm} when we assume \eqref{MainAssumpt} 
with the property that the inequality in \eqref{MainAssumpt} is strict at some point when $(M,g)$ is conformally diffeomorphic to the unit $3$-sphere.
\end{proof}

At this point it remains to prove the existence part in 
Theorem \ref{MainThm} when $p=6$ and $(M,g)$ is conformally diffeomorphic to the $3$-sphere. 
This is the subject of what follows.

\begin{proof}[Proof of the Existence Part in Theorem \ref{MainThm} when $p = 6$. Case 2] We assume that  $(M,g)$ is conformally diffeomorphic to the unit $3$-sphere. 
Without loss of generality we can assume that $M = S^3$ and that $g = \varphi^4g_0$ for some smooth positive function $\varphi > 0$. 
We let $(\beta_\alpha)_\alpha$ be any sequence 
of real numbers such that $\beta_\alpha > 1$ for all $\alpha$ and $\beta_\alpha \to 1$ as $\alpha \to +\infty$. We fix $x_0 \in S^3$ and define the functions 
$\varphi_\alpha: S^3\to\mathbb{R}$ by
\begin{equation}\label{DefVarphialp}
\varphi_\alpha(x) = \frac{\left(3(\beta_\alpha^2-1)\right)^{1/4}}{\varphi\sqrt{2\left(\beta_\alpha - \cos r\right)}}\hskip.1cm ,
\end{equation}
where $r = d_{g_0}(x_0,x)$. We have that
\begin{equation}\label{EqtPhialp}
\Delta_g\varphi_\alpha + \frac{1}{8}S_g\varphi_\alpha = \varphi_\alpha^5
\end{equation}
and that 
\begin{equation}\label{EqtPhialpExtr}
\left(\int_{S^3}\varphi_\alpha^6dv_g\right)^{1/3} = K_3^2\int_{S^3}\left(\vert\nabla\varphi_\alpha\vert^2 + \frac{1}{8}S_g\varphi_\alpha^2\right)dv_g
\end{equation}
for all $\alpha$. A possible reference in book form for \eqref{EqtPhialp} and \eqref{EqtPhialpExtr} is Hebey 
\cite{CIMSBook}. It follows from \eqref{DefVarphialp}-\eqref{EqtPhialpExtr} that
\begin{equation}\label{EnerLevel}
\int_{S^3}\vert\nabla\varphi_\alpha\vert^2dv_g = \frac{1}{K_3^3} + o(1)
\hskip.2cm\hbox{and}\hskip.2cm
\int_{S^3}\varphi_\alpha^6dv_g = \frac{1}{K_3^3} 
\end{equation}
for all $\alpha$, while $\varphi_\alpha \to 0$ in $L^q$ for all $q < 6$ as $\alpha \to +\infty$. If $\lambda = 0$ or $\omega = 0$, then, by \eqref{MainAssumpt}, either 
$a \equiv \frac{1}{8}S_g$ or $a \le \frac{1}{8}S_g$, the inequality being strict at least at one point. In the first case,
by \eqref{EqtPhialp}, for any $\alpha$, 
$(\varphi_\alpha,\frac{1}{q})$ is a solution of \eqref{SWSyst} if $\lambda = 0$ while 
$\left(\varphi_\alpha,\Phi(\varphi_\alpha)\right)$ is a solution of \eqref{SWSyst} if $\lambda > 0$. In the second case, by \eqref{EqtPhialpExtr},
$$\frac{\int_{S^3}\left(\vert\nabla\varphi_\alpha\vert^2+a\varphi_\alpha^2\right)dv_g}{\left(\int_{S^3}\varphi_\alpha^6dv_g\right)^{1/3}} 
< \frac{1}{K_3^2}$$
and it follows from Aubin \cite{Aub1,Aub2} that there exists $u$ smooth and positive such that
$$\Delta_gu + au = u^5\hskip.1cm .$$
Then $(u,\frac{1}{q})$ solves \eqref{SWSyst} if $\lambda = 0$ while $\left(u,\Phi(u)\right)$ solves \eqref{SWSyst} if $\lambda > 0$.  
We may thus assume that $\lambda > 0$ and $\omega^2 > 0$. By \eqref{EnerLevel} there exists $T > 0$, independent 
of $\alpha$, such that $I_6(T\varphi_\alpha) < 0$ for all $\alpha \gg 1$. Let $(t_\alpha)_\alpha$ be any 
sequence in $[0,T]$. By \eqref{LipsType}, since $\Phi(0) = 0$, there holds that $\Phi(t_\alpha\varphi_\alpha) \to 0$ 
in $H^1$ as $\alpha \to +\infty$. By \eqref{DefPhi} and \eqref{EqtDbleIneq} we then get that $\Phi(t_\alpha\varphi_\alpha) \to 0$ 
in $H^{2,q}$ for all $q < 3$ as $\alpha \to +\infty$. In particular, $\Phi(t_\alpha\varphi_\alpha) \to 0$ in $L^\infty$ as $\alpha \to +\infty$, and we can write that
\begin{equation}\label{C0Conv}
\max_{0 \le t \le T}\Vert\Phi(t\varphi_\alpha)\Vert_{L^\infty} \to 0
\end{equation}
as $\alpha \to +\infty$. Since $k\lambda < 1$, it follows from \eqref{C0Conv} that there exists $\varepsilon_0 > 0$ such that for any $\alpha \gg 1$, and any $t \in [0,T]$,
\begin{equation}\label{IneqC0Ex1bis}
\int_M\left(1-q\Phi(t\varphi_\alpha)\right)u_\alpha^2dv_g \ge \left(k\lambda+\frac{\varepsilon_0}{\omega^2}\right) \int_Mu_\alpha^2dv_g
\hskip.1cm .
\end{equation}
Let $\mathcal{H}_6$ be the functional defined on $H^1$ by
\begin{equation}\label{DefF6}
\mathcal{H}_6(u) = \frac{1}{2}\int_{S^3}\vert\nabla u\vert^2dv_g + \frac{1}{2}\int_{S^3}\left(\frac{1}{8}S_g-\varepsilon_0\right)u^2dv_g - 
\frac{1}{6}\int_{S^3}\vert u\vert^6dv_g\hskip.1cm .
\end{equation}
Then, by \eqref{MainAssumpt} and \eqref{IneqC0Ex1bis},
\begin{equation}\label{IneqF6}
\max_{0\le t\le T}I_6(t\varphi_\alpha) \le \max_{0\le t\le T}\mathcal{H}_6(t\varphi_\alpha)
\end{equation}
for all $\alpha \gg 1$. Fix $u_0 = \varphi_\alpha$ for $\alpha \gg 1$, and let $T_0 = T$. Then, for $\varepsilon > 0$ sufficiently small, we 
can write that $I_p(T_0u_0) < 0$ and 
\begin{equation}\label{IneqIpI6}
\max_{0\le t\le T_0}I_p(tu_0) \le \left(1+\delta_\varepsilon\right)\max_{0\le t\le T_0}I_6(tu_0)
\end{equation}
for all $p \in (6-\varepsilon,6)$, where $\delta_\varepsilon \to 0$ as $\varepsilon \to 0$. We have that
\begin{equation}\label{MaxF6}
 \max_{0 \le t \le T_0}\mathcal{H}_6(tu_0) \le \left(\frac{1}{2}-\frac{1}{6}\right) 
 \left(\frac{\int_{S^3}\left(\vert\nabla u_0\vert^2+\left(\frac{1}{8}S_g-\varepsilon_0\right)u_0^2\right)dv_g}{\left(\int_{S^3}u_0^6dv_g\right)^{1/3}}\right)^{3/2}\hskip.1cm .
 \end{equation}
By \eqref{EqtPhialp} and \eqref{EqtPhialpExtr}, since $\varepsilon_0 > 0$, it follows from \eqref{MaxF6} that
\begin{equation}\label{MaxF6Bis}
 \max_{0 \le t \le T_0}\mathcal{H}_6(tu_0) < \frac{1}{3K_3^3}\hskip.1cm .
\end{equation}
As a consequence of \eqref{Eqt2PExist1}, \eqref{IneqF6}, \eqref{IneqIpI6} and \eqref{MaxF6Bis}, we get that 
there exists $\delta_0 > 0$ such that for $0 < \varepsilon \ll 1$ sufficiently small, 
\begin{equation}\label{Proof2ExEqt32ndPart}
\delta_0 \le c_p \le \frac{1}{3K_3^3} - \delta_0
\end{equation}
for all $p \in (6-\varepsilon,6)$. Then we conclude as in case 1 of the proof. This ends the proof of the existence part in Theorem \ref{MainThm}.
\end{proof}

\section{Proof of the uniform bound in Theorem \ref{MainThm} when $p \in (2,6)$}\label{Sect2}

We prove the uniform bounds in Theorem \ref{MainThm} when $p \in (2,6)$. Let $(\omega_\alpha)_\alpha$ be a sequence in $(-\omega_a,\omega_a)$ such that 
$\omega_\alpha \to \omega$ as $\alpha \to +\infty$ for some $\omega \in [-\omega_a,\omega_a]$. Also let 
$p \in (2,6)$ and $\bigl((u_\alpha,v_\alpha)\bigr)_\alpha$ be a sequence 
of smooth positive solutions of \eqref{SWSyst} with phases $\omega_\alpha$. Then,
\begin{equation}\label{SWSystAlpha}
\begin{cases}
\Delta_gu_\alpha + au_\alpha = u_\alpha^{p-1} + \omega_\alpha^2\left(qv_\alpha-1\right)^2u_\alpha\\
\Delta_gv_\alpha + \left(\lambda+q^2u_\alpha^2\right)v_\alpha = qu_\alpha^2
\end{cases}
\end{equation}
for all $\alpha$. Since $u_\alpha > 0$, we get with the second equation in \eqref{SWSystAlpha}, see \eqref{Eqt0LemC1}, that 
$0 \le v_\alpha \le \frac{1}{q}$ for all $\alpha$. Assume by contradiction that
\begin{equation}\label{ContrAssumptSubCptness1}
\max_Mu_\alpha \to +\infty
\end{equation}
as $\alpha \to +\infty$. Let $x_\alpha \in M$ and $\mu_\alpha > 0$ be given by
$$u_\alpha(x_\alpha) = \max_Mu_\alpha = \mu_\alpha^{-2/(p-2)}\hskip.1cm .$$
By \eqref{ContrAssumptSubCptness1}, $\mu_\alpha \to 0$ as $\alpha \to +\infty$. Define $\tilde u_\alpha$ by
$$\tilde u_\alpha(x) = \mu_\alpha^{\frac{2}{p-2}}u_\alpha\left(\exp_{x_\alpha}(\mu_\alpha x)\right)$$
and $g_\alpha$ by $g_\alpha(x) = \left(\exp_{x_\alpha}^\star g\right)(\mu_\alpha x)$ for $x \in B_0(\delta\mu_\alpha^{-1})$, 
where $\delta > 0$ is small. Since $\mu_\alpha \to 0$, we get that $g_\alpha \to \xi$ in $C^2_{loc}(\mathbb{R}^3)$ as $\alpha \to +\infty$. 
Moreover, by \eqref{SWSystAlpha},
\begin{equation}\label{EqtTildeUAlpha}
\Delta_{g_\alpha}\tilde u_\alpha + \mu_\alpha^2\hat a_\alpha\tilde u_\alpha = \tilde u_\alpha^{p-1} 
+ \omega_\alpha^2\mu_\alpha^2\left(q\hat v_\alpha - 1\right)^2\tilde u_\alpha\hskip.1cm ,
\end{equation}
where $\hat a_\alpha$ and $\hat v_\alpha$ are given by
$$\hat a_\alpha(x) = a\left(\exp_{x_\alpha}(\mu_\alpha x)\right)
\hskip.2cm\hbox{and}\hskip.2cm
\hat v_\alpha(x) = v_\alpha\left(\exp_{x_\alpha}(\mu_\alpha x)\right)\hskip.1cm .$$
In addition, $\tilde u_\alpha(0) = 1$ and $0 \le \tilde u_\alpha \le 1$. By \eqref{EqtTildeUAlpha} and standard elliptic theory 
arguments, we can write that, after passing to a subsequence, $\tilde u_\alpha \to u$ in $C^{1,\theta}_{loc}(\mathbb{R}^3)$ 
as $\alpha \to +\infty$, where $u$ is such that $u(0) = 1$ and $0 \le u \le 1$. Then
$$\Delta_\xi u = u^{p-1}$$
in $\mathbb{R}^3$, where $\Delta_\xi$ is the Euclidean Laplacian. It follows that $u$ is actually smooth and positive, and, 
since $2 < p < 6$, we get 
a contradiction with the Liouville result of Gidas and Spruck \cite{GidSpr}. As a conclusion, \eqref{ContrAssumptSubCptness1} is 
not possible and there exists $C > 0$ such that
\begin{equation}\label{ConclContr1}
u_\alpha + v_\alpha \le C
\end{equation}
in $M$ for all $\alpha$. Coming back to \eqref{SWSystAlpha} it follows that the sequences $(u_\alpha)_\alpha$ and $(v_\alpha)_\alpha$ 
are actually bounded in $H^{2,q}$ for all $q$. Pushing one step further the regularity argument they turn out to be bounded in $H^{3,q}$ for all $q$, 
and by the Sobolev embedding theorem we get that they are also 
bounded in $C^{2,\theta}$, $0 < \theta < 1$. This ends the proof of the uniform bounds in Theorem \ref{MainThm} when $p \in (2,6)$. 

\medskip As a remark on the above proof it is necessary to assume that $u_\alpha \not\equiv 0$ since if not the case, when 
$\lambda = 0$, couples like $(0,t)$ solve \eqref{SWSyst} for all $t > 0$. Assuming $u_\alpha \ge 0$, $u_\alpha\not\equiv 0$, we get that 
$u_\alpha > 0$ in $M$ and also that $v_\alpha > 0$ in $M$. An operator like $\Delta_g + h$ is said to be coercive if its energy is a possible norm for $H^1$ 
or, in an equivalent way, if there exists $C > 0$ such that
$$C\int_Mu^2dv_g \le \int_M\left(\vert\nabla u\vert^2+hu^2\right)dv_g$$
for all $u \in H^1$. A complementary lemma is as follows.

\begin{lem}\label{LemNonZer} Let $(\omega_\alpha)_\alpha$ be a sequence in $(-\omega_a,\omega_a)$ such that 
$\omega_\alpha \to \omega$ as $\alpha \to +\infty$ for some $\omega \in [-\omega_a,\omega_a]$, 
$p \in (2,6]$, and $\bigl((u_\alpha,v_\alpha)\bigr)_\alpha$ be a sequence 
of smooth positive solutions of \eqref{SWSystAlpha}. Assume that the operator $\Delta_g + (a-\omega^2)$ is coercive. Let $u_\alpha \to u$ and $v_\alpha \to v$ 
in $C^2$ as $\alpha \to +\infty$. Then $u > 0$, $v > 0$, and $u, v$ are smooth solutions of \eqref{SWSyst}.
\end{lem}

\begin{proof} Assume that $\Delta_g + (a-\omega^2)$ is coercive. Then, for $\varepsilon > 0$ sufficiently small, 
$\Delta_g + (a-\omega^2-\varepsilon)$ is still coercive. Since $u_\alpha > 0$ in $M$ there holds that $0 \le v_\alpha \le \frac{1}{q}$ for all $\alpha$. 
In particular, by \eqref{SWSystAlpha} and the Sobolev inequality, for any $\alpha \gg 1$ sufficiently large,
\begin{eqnarray*} 
&&\int_M\left(\vert\nabla u_\alpha\vert^2+ \left(a-\omega^2-\varepsilon\right)u_\alpha^2\right)dv_g\\
&&\le \int_M\vert\nabla u_\alpha\vert^2dv_g+ \int_Mau_\alpha^2dv_g - \omega_\alpha^2\int_M(qv_\alpha^2-1)^2u_\alpha^2dv_g\\
&&= \int_Mu_\alpha^pdv_g\\
&&\le C\left(\int_M\left(\vert\nabla u_\alpha\vert^2+ \left(a-\omega^2-\varepsilon\right)u_\alpha^2\right)dv_g\right)^{p/2}
\end{eqnarray*}
for some $C > 0$ independent of $\alpha$. This implies $u > 0$ and then $v > 0$. The lemma follows.
\end{proof}

As a remark, Lemma \ref{LemNonZer} does not hold anymore if we allow $\Delta_g + (a-\omega^2)$ not to be coercive. Suppose $\lambda > 0$, $a >0$ is a positive constant, and let 
$(\varepsilon_\alpha)_\alpha$ be a sequence of positive real numbers such that $\varepsilon_\alpha \to 0$ as 
$\alpha \to +\infty$. Assuming $a$ is constant we have that $\omega_a^2 = a$. Let $u_\alpha = \varepsilon_\alpha$ and
$$v_\alpha = \frac{q\varepsilon_\alpha^2}{\lambda + q^2\varepsilon_\alpha^2}\hskip.1cm .$$
Then $u_\alpha \to 0$ and 
$v_\alpha \to 0$ in $C^2$ as $\alpha \to +\infty$, and we do have that $(u_\alpha,v_\alpha)$ solves \eqref{SWSystAlpha} , where 
$$\omega_\alpha^2 = \frac{1}{(qv_\alpha-1)^2}\left(\omega_a^2-\varepsilon_\alpha^{p-2}\right)\hskip.1cm .$$
Noting that $\omega_\alpha \to \omega_a$ as $\alpha \to +\infty$, the construction provides a counter example to Lemma 
\ref{LemNonZer} when $\omega^2 = a$. As an independent remark, $\Delta_g + (a-\omega^2)$ is automatically coercive when $\omega \in (-\omega_a,\omega_a)$. 
Also we may allow $\omega=\pm\omega_a$ if $\Delta_g + (a-\min_Ma)$ is coercive, and this is automatically the case if $a$ is nonconstant. 

\section{Sharp blow-up estimates when $p=6$}\label{blowupest}

In what follows we let $(M,g)$ be a smooth compact $3$-dimensional Riemannian manifold, 
$a > 0$ be a smooth positive function in $M$, and $(\omega_\alpha)_\alpha$ be a sequence in $(-\omega_a,\omega_a)$ such that 
$\omega_\alpha \to \omega$ as $\alpha \to +\infty$ for some $\omega \in [-\omega_a,\omega_a]$, where $\omega_a$ is as in \eqref{DefOm0a}. Also we let 
$\bigl((u_\alpha,v_\alpha)\bigr)_\alpha$ be a sequence 
of smooth positive solutions of \eqref{SWSyst} with phases $\omega_\alpha$ and $p = 6$. Namely,
\begin{equation}\label{SWSystAlphaCrit}
\begin{cases}
\Delta_gu_\alpha + au_\alpha = u_\alpha^5 + \omega_\alpha^2\left(qv_\alpha-1\right)^2u_\alpha\\
\Delta_gv_\alpha + \left(\lambda+q^2u_\alpha^2\right)v_\alpha = qu_\alpha^2
\end{cases}
\end{equation}
for all $\alpha$. Since $u_\alpha > 0$, we get with the second equation in \eqref{SWSystAlphaCrit}, see \eqref{Eqt0LemC1}, that 
$0 \le v_\alpha \le \frac{1}{q}$ for all $\alpha$. In particular, if we let
\begin{equation}\label{Defhalpha}
h_\alpha = a - \omega_\alpha^2\left(qv_\alpha-1\right)^2\hskip.1cm ,
\end{equation}
then $\Vert h_\alpha\Vert_{L^\infty} \le C$ for all $\alpha$, where $C > 0$ is independent of $\alpha$. 
Assume by contradiction that
\begin{equation}\label{ContrAssumptSubCptness}
\max_Mu_\alpha \to +\infty
\end{equation}
as $\alpha \to +\infty$. In what follows we let $(x_\alpha)_\alpha$ be a sequence of points in $M$, and $(\rho_\alpha)_\alpha$ be a 
sequence of positive real numbers, $0 < \rho_\alpha < i_g/7$ for all $\alpha$, where $i_g$ is the injectivity radius 
of $(M,g)$. We assume that the $x_\alpha$'s and $\rho_\alpha$'s satisfy
\begin{equation}\label{hypCdts}
\begin{cases}
\nabla u_\alpha(x_\alpha) = 0\hskip.1cm\hbox{for all}\hskip.1cm \alpha ,\\
d_g(x_\alpha,x)^{\frac{1}{2}}u_\alpha(x) \le C\hskip.1cm\hbox{for all}\hskip.1cm x\in B_{x_\alpha}(7\rho_\alpha)\hskip.1cm\hbox{and all}\hskip.1cm \alpha\hskip.1cm ,\\
\lim_{\alpha \to +\infty}\rho_\alpha^{\frac{1}{2}}\sup_{B_{x_\alpha}(6\rho_\alpha)}u_\alpha(x) = +\infty\hskip.1cm .
\end{cases}
\end{equation}
We let $\mu_\alpha$ be given by
\begin{equation}\label{DefMualpha}
\mu_\alpha = u_\alpha(x_\alpha)^{-2}\hskip.1cm .
\end{equation}
Since the $h_\alpha$'s in \eqref{Defhalpha} are $L^\infty$-bounded we can apply the asymptotic analysis in Druet and Hebey \cite{DruHebEinstLich} and 
Druet, Hebey and V\'etois \cite{DruHebVet}. Closely related arguments were first developed by Schoen \cite{SchNotes}, and then 
by Druet \cite{Dru2} and Li and Zhu \cite{LiZhu} assuming $C^1$-convergences of the potentials. Assuming \eqref{hypCdts}, and 
coming back to the analysis in Druet, Hebey and V\'etois \cite{DruHebVet}, 
we can write that $\frac{\rho_\alpha}{\mu_\alpha} \to +\infty$ as 
$\alpha \to +\infty$ and that
\begin{equation}\label{Limualpha}
\mu_\alpha^{\frac{1}{2}}u_\alpha\left(\exp_{x_\alpha}(\mu_\alpha x)\right) \to \left(1 + \frac{\vert x\vert^2}{3}\right)^{-\frac{1}{2}}
\end{equation}
in $C^1_{loc}(\mathbb{R}^3)$ as $\alpha \to +\infty$, where $\mu_\alpha$ is as in \eqref{DefMualpha}. In particular we have that 
$\mu_\alpha \to 0$ 
as $\alpha \to +\infty$. 
Now we define $\varphi_\alpha : \left(0,\rho_\alpha\right)\mapsto {\mathbb R}^+$ by 
\begin{equation}\label{eq3.1}
\varphi_\alpha(r)= \frac{1}{\left\vert \partial B_{x_\alpha}\left(r\right)\right\vert_g} 
\int_{\partial B_{x_\alpha}\left(r\right)} u_\alpha d\sigma_g\hskip.1cm ,
\end{equation}
where $\left\vert \partial B_{x_\alpha}\left(r\right)\right\vert_g$ is the volume of the sphere of center $x_\alpha$ 
and radius $r$ for the induced metric. 
Let $\Lambda = 2\sqrt{3}$. 
We define $r_\alpha\in \left[\Lambda\mu_\alpha, \rho_\alpha\right]$ by 
\begin{equation}\label{eq3.3}
r_\alpha = \sup\left\{r\in \left[\Lambda\mu_\alpha, \rho_\alpha\right]\hbox{ s.t. } \left(s^{\frac{1}{2}}\varphi_\alpha(s)\right)' \le 0 
\hskip.1cm\hbox{in}\hskip.1cm\left[\Lambda\mu_\alpha,r\right]\right\}\hskip.1cm .
\end{equation} 
It follows from \eqref{Limualpha} that 
\begin{equation}\label{eq3.4}
\frac{r_\alpha}{\mu_\alpha}\to +\infty
\end{equation}
as $\alpha \to +\infty$, while the definition of $r_\alpha$ gives that 
\begin{equation}\label{eq3.5}
r^{\frac{1}{2}}\varphi_\alpha \hbox{ is non-increasing in }\left[\Lambda\mu_\alpha,r_\alpha\right]
\end{equation}
and that 
\begin{equation}\label{eq3.6}
\left(r^{\frac{1}{2}}\varphi_\alpha(r)\right)'\left(r_\alpha\right)=0\hbox{ if }r_\alpha<\rho_\alpha\hskip.1cm.
\end{equation}
We prove that the following sharp asymptotic estimates on the $u_\alpha$'s in \eqref{SWSystAlphaCrit} hold true.

\begin{lem}\label{SharpEst} Let $(M,g)$ be a smooth compact Riemannian $3$-dimensional manifold, 
and $\bigl((u_\alpha,v_\alpha)\bigr)_\alpha$ be a sequence 
of smooth positive solutions of \eqref{SWSystAlphaCrit} such that \eqref{ContrAssumptSubCptness} holds true.
Let $(x_\alpha)_\alpha$ and $(\rho_\alpha)_\alpha$ be such that \eqref{hypCdts} hold true, and let $R \ge 6$ be such that 
$Rr_\alpha \le 6\rho_\alpha$ for all $\alpha \gg 1$. There exists $C > 0$ such that, after passing to a subsequence,
\begin{equation}\label{SharpEstCtrl}
u_\alpha(x) + d_g(x_\alpha,x)\left\vert\nabla u_\alpha(x)\right\vert \le 
C\mu_\alpha^{\frac{1}{2}}d_g(x_\alpha,x)^{-1}
\end{equation}
for all $x \in B_{x_\alpha}(\frac{R}{2}r_\alpha)\backslash\left\{x_\alpha\right\}$ and all $\alpha$, where
$\mu_\alpha$ is as in \eqref{DefMualpha}, and where $r_\alpha$ 
is as in \eqref{eq3.3}.
\end{lem}

\begin{proof}[Proof of Lemma \ref{SharpEst}] Given $R > 0$ we define
\begin{equation}\label{DefEta}
\eta_{R,\alpha} = \sup_{B_{x_\alpha}(Rr_\alpha)\backslash B_{x_\alpha}(\frac{1}{R}r_\alpha)}u_\alpha\hskip.1cm .
\end{equation}
We prove that there exist $C, C^\prime > 0$ such that
\begin{equation}\label{Eqt1Lem4}
u_\alpha(x) \le C\left(\mu_\alpha^{\frac{1}{2}}d_g(x_\alpha,x)^{-1} + \eta_{R,\alpha}\right)
\end{equation}
for all $x \in B_{x_\alpha}(\frac{R}{2}r_\alpha)\backslash\{x_\alpha\}$ and all $\alpha$, and such that
\begin{equation}\label{Eqt2Lem4}
\eta_{R,\alpha} \le C^\prime\mu_\alpha^{\frac{1}{2}}r_\alpha^{-1}
\end{equation}
for all $\alpha$. Let $R^\prime \ge 6$ be given. The Harnack inequality in Druet, Hebey and V\'etois \cite{DruHebVet} can be stated in the 
following way: there 
exists $C>1$ such that for any sequence $(s_\alpha)_\alpha$ of positive real numbers satisfying that $s_\alpha > 0$ 
and $R^\prime s_\alpha \le 6\rho_\alpha$ for all $\alpha$, there holds
\begin{equation}\label{HarnackIneq}
s_\alpha \left\Vert \nabla u_\alpha \right\Vert_{L^\infty\left(\Omega_\alpha\right)}\le C \sup_{\Omega_\alpha}u_\alpha 
\le C^2 \inf_{\Omega_\alpha} u_\alpha\hskip.1cm ,
\end{equation}
where $\Omega_\alpha = B_{x_\alpha}\left(R^\prime s_\alpha\right)\setminus B_{x_\alpha}\left(\frac{1}{R^\prime}s_\alpha\right)$.
Lemma \ref{SharpEst} follows from \eqref{Eqt1Lem4}, \eqref{Eqt2Lem4}, and \eqref{HarnackIneq} 
in order to get the gradient part in \eqref{SharpEstCtrl}. We start with the proof of \eqref{Eqt1Lem4}.  For this aim 
we let $(y_\alpha)_\alpha$ be an arbitrary sequence in $B_{x_\alpha}(\frac{R}{2}r_\alpha)\backslash\{x_\alpha\}$, and prove that 
there exists $C > 0$ such that, up to a subsequence,
\begin{equation}\label{Eqt4Lem4}
u_\alpha(y_\alpha) \le C\left(\mu_\alpha^{\frac{1}{2}}d_g(x_\alpha,y_\alpha)^{-1}+\eta_{R,\alpha}\right)\hskip.1cm .
\end{equation}
As a preliminary remark one can note that \eqref{Eqt4Lem4} directly follows from \eqref{Limualpha} if $d_g(x_\alpha,y_\alpha) = O(\mu_\alpha)$. 
By \eqref{HarnackIneq} we may then assume that
\begin{equation}\label{Eqt5Lem4}
\lim_{\alpha\to+\infty}\frac{1}{\mu_\alpha}d_g(x_\alpha,y_\alpha) = +\infty
\hskip.2cm\hbox{and}\hskip.2cm
\lim_{\alpha \to +\infty}\frac{1}{r_\alpha}d_g(x_\alpha,y_\alpha) = 0\hskip.1cm .
\end{equation}
Without loss of generality, since the $\Vert h_\alpha\Vert_{L^\infty}$'s are bounded, we can assume that, up to a subsequence, 
$\Vert h_\alpha\Vert_{L^\infty} \to \Lambda$ as $\alpha \to +\infty$ for some $\Lambda \ge 0$, where the $h_\alpha$'s are as in \eqref{Defhalpha}.
Now we let $k > 1$ be such that $k\Lambda \not\in \hbox{Sp}(\Delta_g)$, where $\hbox{Sp}(\Delta_g)$ is the spectrum of $\Delta_g$, and let $G$ be the Green's 
function of $\Delta_g - k\Lambda$. Then, see, for instance, Robert \cite{RobGreen}, there are positive constants $C_1 > 1$ and $C_2, C_3 > 0$ such that
\begin{equation}\label{Eqt6Lem4}
\begin{split}
&\frac{1}{C_1}d_g(x,y)^{-1} - C_2 \le G(x,y) \le C_1 d_g(x,y)^{-1}\hskip.1cm ,\hskip.1cm\hbox{and}\\
&\vert\nabla G(x,y)\vert \le C_3d_g(x,y)^{-2}
\end{split}
\end{equation}
for all $x \not= y$. By \eqref{Eqt6Lem4} there exists $\delta > 0$ such that $G \ge 0$ in $B_{x_\alpha}(\delta r_\alpha)$ for all 
$\alpha$. By \eqref{Eqt5Lem4}, $y_\alpha \in B_{x_\alpha}(\frac{\delta}{2}r_\alpha)$ for $\alpha \gg 1$, and by the Green's representation formula,
\begin{equation}\label{Eqt7Lem4}
\begin{split}
u_\alpha(y_\alpha)
&= \int_{B_{x_\alpha}(\delta r_\alpha)}G(y_\alpha,x)\left(\Delta_gu_\alpha - k\Lambda u_\alpha\right)(x)dv_g(x)\\
&+ \int_{\partial B_{x_\alpha}(\delta r_\alpha)}G(y_\alpha,x)\left(\partial_\nu u_\alpha\right)(x)d\sigma_g(x)\\
&- \int_{\partial B_{x_\alpha}(\delta r_\alpha)}\left(\partial_\nu G(y_\alpha,x)\right)u_\alpha(x)d\sigma_g(x)\hskip.1cm ,
\end{split}
\end{equation}
where $\nu$ is the unit outward normal to $\partial B_{x_\alpha}(\delta r_\alpha)$. Since $k > 1$, and 
$\Vert h_\alpha\Vert_{L^\infty}\to\Lambda$ as $\alpha \to +\infty$,
$$\Delta_gu_\alpha - k\Lambda u_\alpha \le u_\alpha^5$$
and since $G \ge 0$ in $B_{x_\alpha}(\delta r_\alpha)$ we get with \eqref{Eqt6Lem4} that
\begin{equation}\label{Eqt8Lem4}
\begin{split}
&\int_{B_{x_\alpha}(\delta r_\alpha)}G(y_\alpha,x)\left(\Delta_gu_\alpha - k\Lambda u_\alpha\right)(x)dv_g(x)\\
&\le C \int_{B_{x_\alpha}(\delta r_\alpha)}d_g(y_\alpha,x)^{-1}u_\alpha(x)^5dv_g(x)\hskip.1cm .
\end{split}
\end{equation}
Also, by \eqref{HarnackIneq} and \eqref{Eqt6Lem4}, we have that
\begin{equation}\label{Eqt9Lem4}
\begin{split}
&\int_{\partial B_{x_\alpha}(\delta r_\alpha)}G(y_\alpha,x)\left\vert\partial_\nu u_\alpha(x)\right\vert d\sigma_g(x)
\le C\eta_{R,\alpha}\hskip.1cm ,\hskip.1cm\hbox{and}\\
&\int_{\partial B_{x_\alpha}(\delta r_\alpha)}\left\vert\partial_\nu G(y_\alpha,x)\right\vert u_\alpha(x)d\sigma_g(x)
\le C\eta_{R,\alpha}
\end{split}
\end{equation}
for some $C > 0$. Combining \eqref{Eqt7Lem4}--\eqref{Eqt9Lem4}, we get that
\begin{equation}\label{Eqt10Lem4}
u_\alpha(y_\alpha) 
\le  C\int_{B_{x_\alpha}(\delta r_\alpha)}d_g(y_\alpha,x)^{-1}u_\alpha^5(x)dv_g(x) + C\eta_{R,\alpha}\hskip.1cm .
\end{equation}
Following Druet, Hebey and V\'etois \cite{DruHebVet}, there holds that
\begin{equation}\label{WeakEst}
u_\alpha(x) \le C\left(\mu_\alpha^{1/10}d_g(x_\alpha,x)^{-3/5}+\eta_{R,\alpha}r_\alpha^{2/5}d_g(x_\alpha,x)^{-2/5}\right)
\end{equation}
for all $x \in B_{x_\alpha}(Rr_\alpha)\backslash\{x_\alpha\}$ and all $\alpha$, where $C > 0$ does not depend on 
$x$ and $\alpha$. In particular, we get with \eqref{hypCdts}, \eqref{Limualpha}, \eqref{Eqt5Lem4}, and \eqref{WeakEst}, that
\begin{equation}\label{EqtDHV1}
\int_{B_{x_\alpha}(\delta r_\alpha)}d_g(y_\alpha,x)^{-1}u_\alpha^5(x)dv_g(x)
= O\left(\mu_\alpha^{\frac{1}{2}}d_g(x_\alpha,y_\alpha)^{-1}\right) + O\left(\eta_{R,\alpha}\right)\hskip.1cm .
\end{equation}
By \eqref{Eqt10Lem4} and \eqref{EqtDHV1}, we obtain \eqref{Eqt4Lem4}. In particular, \eqref{Eqt1Lem4} holds true. Now it remains to prove 
\eqref{Eqt2Lem4}. By \eqref{eq3.5}, for any $\eta \in (0,1)$,
$$\left(\eta r_\alpha\right)^{\frac{1}{2}}\varphi_\alpha(\eta r_\alpha) \ge r_\alpha^{\frac{1}{2}}\varphi_\alpha(r_\alpha)$$
for all $\alpha \gg 1$, where $\varphi_\alpha$ is as in \eqref{eq3.1}. 
By \eqref{HarnackIneq}, there exists $C > 1$ such that
\begin{equation}\label{Eqt1Lem3}
\frac{1}{C}\sup_{B_{x_\alpha}(Rs_\alpha)\backslash B_{x_\alpha}(\frac{1}{R}s_\alpha)}u_\alpha
\le \varphi_\alpha(s_\alpha) \le C\inf_{B_{x_\alpha}(Rs_\alpha)\backslash B_{x_\alpha}(\frac{1}{R}s_\alpha)}u_\alpha
\end{equation}
for all $0 < s_\alpha \le r_\alpha$ and all $\alpha$. 
By \eqref{Eqt1Lem3} we then get that
$$\frac{1}{C}r_\alpha^{\frac{1}{2}}\eta_{R,\alpha} 
\le (\eta r_\alpha)^{\frac{1}{2}}\sup_{\partial B_{x_\alpha}(\eta r_\alpha)}u_\alpha\hskip.1cm .$$
Assuming \eqref{Eqt1Lem4} it follows that
$$\frac{1}{C}\eta_{R,\alpha} \le \eta^{\frac{1}{2}}\left(\mu_\alpha^{\frac{1}{2}}(\eta r_\alpha)^{-1} + \eta_{R,\alpha}\right)$$
and if we choose $\eta \in (0,1)$ sufficiently small such that $C\eta^{\frac{1}{2}} \le \frac{1}{2}$, we obtain that
$$\eta_{R,\alpha} \le \eta^{2-n}\mu_\alpha^{\frac{1}{2}}r_\alpha^{-1}\hskip.1cm .$$
In particular, \eqref{Eqt2Lem4} holds true. This ends the proof of the lemma.
\end{proof}

Now that we have Lemma \ref{SharpEst} we prove that the following fundamental asymptotic estimate holds true.

\begin{lem}\label{LemSharpAsypts} Let $(M,g)$ be a smooth compact Riemannian $3$-dimensional manifold and 
$\bigl((u_\alpha,v_\alpha)\bigr)_\alpha$ be a sequence 
of smooth positive solutions of \eqref{SWSystAlphaCrit} such that \eqref{ContrAssumptSubCptness} holds true. 
Let $(x_\alpha)_\alpha$ and $(\rho_\alpha)_\alpha$ be such that \eqref{hypCdts} holds true. Assume $r_\alpha \to 0$ 
as $\alpha \to +\infty$, where $r_\alpha$ is as in \eqref{eq3.3}. Then $\rho_\alpha = O\left(r_\alpha\right)$ and
\begin{equation}\label{EqtLemSharpAspt}
r_\alpha\mu_\alpha^{-\frac{1}{2}}u_\alpha\left(\exp_{x_\alpha}(r_\alpha x)\right) \to 
\frac{\sqrt{3}}{\vert x\vert} + \mathcal{H}(x)
\end{equation}
in $C^2_{loc}\left(B_0(2)\backslash\{0\}\right)$ as $\alpha \to +\infty$, where $\mu_\alpha$ is as in \eqref{DefMualpha}, 
and $\mathcal{H}$ is a harmonic function in $B_0(2)$ which satisfies that $\mathcal{H}(0) = 0$. 
\end{lem}

\begin{proof}[Proof of Lemma \ref{LemSharpAsypts}] Let $R \ge 6$ be such that $Rr_\alpha \le 6\rho_\alpha$ for $\alpha \gg 1$. 
In what follows we assume that $r_\alpha\to 0$ as $\alpha\to +\infty$. For $x\in B_0(3)$ we set
\begin{eqnarray*}
\tilde{u}_\alpha(x) &=& r_\alpha \mu_\alpha^{-\frac{1}{2}}u_\alpha\left(\exp_{x_\alpha}\left(r_\alpha x\right)\right)\hskip.1cm,\\
g_\alpha(x)&=& \left(\exp_{x_\alpha}^\star g\right)\left(r_\alpha x\right)\hskip.1cm ,\hskip.1cm\hbox{and}\\
\tilde{h}_\alpha(x)&=& h_\alpha\left(\exp_{x_\alpha}(r_\alpha x)\right)\hskip.1cm,
\end{eqnarray*}
where $h_\alpha$ is as in \eqref{Defhalpha}. Since $r_\alpha\to 0$ as $\alpha\to +\infty$, we have that $\tilde{g}_\alpha\to \xi$ in 
$C^2_{loc}(\mathbb{R}^3)$ as $\alpha\to +\infty$, where $\xi$ is the Euclidean metric. 
Thanks to Lemma \ref{SharpEst} we also have that 
\begin{equation}\label{eq5.1}
\left\vert\tilde{u}_\alpha(x)\right\vert \le C \left\vert x\right\vert^{-1}
\end{equation}
in $B_0(\frac{R}{2})\backslash\{0\}$. By \eqref{SWSystAlphaCrit}, \eqref{eq3.4}, and thanks to standard elliptic theory we 
can write that, after passing to a subsequence, 
$\tilde{u}_\alpha \to \tilde{u}$ in $C^2_{loc}\left(B_0(\frac{R}{2})\backslash\{0\}\right)$ as $\alpha \to +\infty$, where $\tilde{u}$ satisfies 
$\Delta\tilde{u} = 0$ in $B_0(\frac{R}{2})\backslash\{0\}$. By \eqref{eq5.1}, $\left\vert\tilde{u}(x)\right\vert \le C \left\vert x\right\vert^{-1}$ 
in $B_0(\frac{R}{2})\backslash\{0\}$. Thus we can write that 
\begin{equation}\label{eq5.6}
\tilde{u}(x) = \frac{\Lambda}{\left\vert x\right\vert}+ \mathcal{H}(x)
\end{equation}
where $\Lambda \ge 0$ and $\mathcal{H}$ satisfies 
$\Delta\mathcal{H}=0$ in $B_0(\frac{R}{2})$. 
In order to see that $\Lambda = \sqrt{3}$, it is sufficient to integrate the equation satisfied by the $\tilde{u}_\alpha$'s 
in $B_0(1)$. Then
\begin{equation}\label{Eqt1LS}
-\int_{\partial B_0(1)} \partial_\nu\tilde{u}_\alpha d\sigma_{g_\alpha}
= \left(\frac{\mu_\alpha}{r_\alpha}\right)^2\int_{B_0(1)}\tilde{u}_\alpha^5dv_{g_\alpha}
 -r_\alpha^2 \int_{B_0(1)} \tilde{h}_\alpha\tilde{u}_\alpha dv_{g_\alpha}\hskip.1cm ,
 \end{equation}
where $\nu$ is the unit outward normal derivative to $\partial B_0(1)$. 
By \eqref{eq5.1}, the $\tilde{u}_\alpha$'s are bounded in $L^1\left(B_0(1)\right)$. 
Changing $x$ into $\frac{\mu_\alpha}{r_\alpha}x$, thanks to \eqref{Limualpha} and Lemma \ref{SharpEst}, we also have that
\begin{equation}\label{Eqt3LS}
\lim_{\alpha\to+\infty}\left(\frac{\mu_\alpha}{r_\alpha}\right)^2\int_{B_0(1)}\tilde{u}_\alpha^5 dv_{g_\alpha}
= \int_{\mathbb{R}^3}\left(\frac{1}{1+\frac{\vert x\vert^2}{3}}\right)^{5/2}dx
= \sqrt{3}\omega_2\hskip.1cm .
\end{equation}
Noting that by \eqref{eq5.6},
\begin{equation}\label{Eqt4LS}
\lim_{\alpha\to+\infty}\int_{\partial B_0(1)}\partial_\nu\tilde{u}_\alpha d\sigma_{g_\alpha} = - \omega_2\Lambda\hskip.1cm ,
\end{equation}
we get that $\Lambda = \sqrt{3}$ by combining \eqref{Eqt1LS}--\eqref{Eqt4LS}. Now we prove that $\mathcal{H}(0) = 0$. 
In what follows we let $X_\alpha$ be the $1$-form given by
\begin{equation}\label{DefX}
X_\alpha(x) = \left(1 - \frac{1}{12}Rc _g^\sharp(x)\left(\nabla f_\alpha(x),\nabla f_\alpha(x)\right)\right)\nabla f_\alpha(x)
\hskip.1cm ,
\end{equation}
where $f_\alpha(x) = \frac{1}{2}d_g(x_\alpha,x)^2$ and, in local coordinates, $(Rc_g^\sharp)^{ij} = g^{i\mu}g^{j\nu}R_{\mu\nu}$, 
where the $R_{ij}$'s are the components of the Ricci curvature $Rc_g$ of $g$. We adopt the notations that $A^\sharp$ is the musical isomorphism 
applied to $A$, and that $X(\nabla u) = (X,\nabla u)$ for $X$ a $1$-form and $u$ a function. By the Pohozaev 
identity in Druet and Hebey \cite{DruHebStrongCoup1}, 
that we apply to the $u_\alpha$'s in $B_{x_\alpha}(r_\alpha)$ with the above choice of $X_\alpha$, we have that
\begin{equation}\label{eq5.7}
\begin{split}
& \int_{B_{x_\alpha}(r_\alpha)} X_\alpha(\nabla u_\alpha) h_\alpha u_\alpha dv_g 
+ \frac{1}{12} \int_{B_{x_\alpha}(r_\alpha)} \left(\Delta_g\hbox{div}_g X_\alpha\right)u_\alpha^2dv_g \\
&+ \frac{1}{6}  \int_{B_{x_\alpha}(r_\alpha)}\left(\hbox{div}_g X_\alpha\right) h_\alpha u_\alpha^2dv_g 
= Q_{1,\alpha} + Q_{2,\alpha} + Q_{3,\alpha}\hskip.1cm ,
\end{split}
\end{equation}
where 
\begin{eqnarray*}
Q_{1,\alpha}&=& \frac{1}{6} \int_{\partial B_{x_\alpha}(r_\alpha)} \left(\hbox{div}_g X_\alpha\right) (\partial_\nu u_\alpha) u_\alpha d\sigma_g \\
&&- \int_{\partial B_{x_\alpha}(r_\alpha)} \left(\frac{1}{2} X_\alpha(\nu)\vert \nabla u_\alpha\vert^2 - X_\alpha(\nabla u_\alpha)\partial_\nu u_\alpha\right) d\sigma_g\hskip.1cm,
\end{eqnarray*}
$$Q_{2,\alpha} = -\int_{B_{x_\alpha}(r_\alpha)} \left(\nabla X_\alpha-\frac{1}{3}\left(\hbox{div}_g X_\alpha\right) g\right)^\sharp
\left(\nabla u_\alpha,\nabla u_\alpha\right) dv_g\hskip.1cm,$$
$$Q_{3,\alpha}= \frac{1}{6}\int_{\partial B_{x_\alpha}(r_\alpha)} X_\alpha\left(\nu\right) u_\alpha^6 d\sigma_g 
 - \frac{1}{12} \int_{\partial B_{x_\alpha}(r_\alpha)} \left(\partial_\nu \left(\hbox{div}_g X_\alpha\right)\right) u_\alpha^2 d\sigma_g\hskip.1cm,$$
and $\nu$ is the unit outward normal derivative to $\partial B_{x_\alpha}(r_\alpha)$. We have that 
\begin{equation}\label{Eqt1TL1}
\left(\nabla X_\alpha\right)_{ij} - \frac{1}{n}\left(\hbox{div}_gX_\alpha\right)g_{ij} = O\left(d_g(x_\alpha,x)^2\right)
\end{equation}
for all $i, j$. By Lemma \ref{SharpEst} and \eqref{Eqt1TL1} we then get that
\begin{equation}\label{IntCalc1}
\begin{split}
\left\vert Q_{2,\alpha}\right\vert
&\le C \int_{B_{x_\alpha}(r_\alpha)}d_g(x_\alpha,x)^2\vert\nabla u_\alpha(x)\vert^2dv_g(x)\\
&\le C\mu_\alpha\int_{B_{x_\alpha}(r_\alpha)}d_g(x_\alpha,x)^{-2}dv_g(x)\\
&\le C\mu_\alpha r_\alpha\hskip.1cm .
\end{split}
\end{equation}
Similarly, 
\begin{equation}\label{Eqt1LT2}
\begin{split}
&\vert X_\alpha(x)\vert = O\left(d_g(x_\alpha,x)\right)\hskip.1cm ,\\
&\hbox{div}_gX_\alpha(x) = 3 + O\left(d_g(x_\alpha,x)^2\right)\hskip.1cm ,\hskip.1cm\hbox{and}\\
&\Delta_g\left(\hbox{div}_gX_\alpha\right)(x) = \frac{3}{2}S_g(x_\alpha) + O\left(d_g(x_\alpha,x)\right)\hskip.1cm .
\end{split}
\end{equation}
By \eqref{Eqt1LT2} and Lemma \ref{SharpEst}, we then get that
\begin{equation}\label{eq5.8}
\begin{split}
& \int_{B_{x_\alpha}(r_\alpha)} X_\alpha(\nabla u_\alpha) h_\alpha u_\alpha dv_g 
+ \frac{1}{12} \int_{B_{x_\alpha}(r_\alpha)} \left(\Delta_g\hbox{div}_g X_\alpha\right)u_\alpha^2dv_g \\
&+ \frac{1}{6}  \int_{B_{x_\alpha}(r_\alpha)}\left(\hbox{div}_g X_\alpha\right) h_\alpha u_\alpha^2dv_g 
= O\left(\mu_\alpha r_\alpha\right)\hskip.1cm ,
\end{split}
\end{equation}
 and since there also holds that $\vert\nabla\left(\hbox{div}_gX_\alpha\right)(x)\vert = O\left(d_g(x_\alpha,x)\right)$, we can write 
 in addition that
\begin{equation}\label{eq5.8bis}
\left\vert Q_{3,\alpha}\right\vert = O\left(\mu_\alpha^3r_\alpha^{-3}\right) + O\left(\mu_\alpha r_\alpha\right)\hskip.1cm .
\end{equation}
Combining \eqref{eq5.7}, \eqref{IntCalc1}, \eqref{eq5.8}, and \eqref{eq5.8bis} we get that 
\begin{equation}\label{FrstComp1}
\begin{split}
&\frac{1}{6} \int_{\partial B_{x_\alpha}(r_\alpha)} \left(\hbox{div}_g X_\alpha\right) (\partial_\nu u_\alpha) u_\alpha d\sigma_g \\
&- \int_{\partial B_{x_\alpha}(r_\alpha)} \left(\frac{1}{2} X_\alpha(\nu)\vert \nabla u_\alpha\vert^2 - X_\alpha(\nabla u_\alpha)\partial_\nu u_\alpha\right) d\sigma_g\\
&= O\left(\mu_\alpha^3r_\alpha^{-3}\right) + O\left(\mu_\alpha r_\alpha\right)\hskip.1cm.
\end{split}
\end{equation}
Since $\tilde{u}_\alpha \to \tilde{u}$ in $C^2_{loc}\left(B_0(\frac{R}{2})\backslash\{0\}\right)$ as $\alpha \to +\infty$, where $\tilde{u}$ satisfies 
$\Delta\tilde{u} = 0$ and $\left\vert\tilde{u}(x)\right\vert \le C \left\vert x\right\vert^{-1}$ 
in $B_0(\frac{R}{2})\backslash\{0\}$, and by \eqref{eq5.6}, we independently get that 
\begin{equation}\label{FrstComp2}
\begin{split}
&\frac{1}{6} \int_{\partial B_{x_\alpha}(r_\alpha)} \left(\hbox{div}_g X_\alpha\right) (\partial_\nu u_\alpha) u_\alpha d\sigma_g \\
&- \int_{\partial B_{x_\alpha}(r_\alpha)} \left(\frac{1}{2} X_\alpha(\nu)\vert \nabla u_\alpha\vert^2 - X_\alpha(\nabla u_\alpha)\partial_\nu u_\alpha\right) d\sigma_g\\
&= \left(\frac{3}{2}\omega_2\Lambda\mathcal{H}(0) + o(1)\right)\frac{\mu_\alpha}{r_\alpha}\hskip.1cm .
\end{split}
\end{equation}
Combining \eqref{FrstComp1} and \eqref{FrstComp2}, we get with \eqref{eq3.4} that $\mathcal{H}(0) = 0$. At this point it remains to prove that 
$\rho_\alpha = O\left(r_\alpha\right)$. We proceed by contradiction and assume that $r_\alpha\rho_\alpha^{-1} \to 0$ as 
$\alpha \to +\infty$. Then \eqref{eq5.6} holds in $B_0(R)\backslash\{0\}$ for all $R$, and 
$r_\alpha < \rho_\alpha$ for $\alpha \gg 1$. In particular, we get with \eqref{eq3.6} and \eqref{eq5.6} that 
$(r^{1/2}\varphi(r))^\prime(1) = 0\hskip.1cm ,$
where
\begin{eqnarray*}
\varphi(r) 
&=& \frac{1}{\omega_2r^2}\int_{\partial B_0(r)}\tilde u d\sigma\\
&=& \frac{\Lambda}{r^{n-2}} + \mathcal{H}(0)\hskip.1cm .
\end{eqnarray*}
Since $\mathcal{H}(0) = 0$, it follows that $\Lambda = 0$, and this is impossible since $\Lambda = \sqrt{3}$. 
Lemma \ref{LemSharpAsypts} is proved.
\end{proof}

\section{Proof of the uniform bound in Theorem \ref{MainThm} when $p = 6$}\label{Sect3}

We prove that the uniform bound in the theorem holds true when $p = 6$. For this aim we use the analysis developed in 
Section \ref{blowupest} to prove that blow-up points are isolated. Then we use 
phase compensation, and the positive mass theorem of Schoen, and Yau \cite{SchYau1} 
(see also Schoen and Yau \cite{SchYau2,SchYau3} Witten \cite{Wit}) to prove that there are no blow-up points when we assume \eqref{MainAssumpt} 
with the property that \eqref{MainAssumpt} is strict at least at one point if $(M,g)$ is conformally diffeomorphic to the unit $3$-sphere 
and $\omega\lambda = 0$.

\medskip Here again we let $(M,g)$ be a smooth compact $3$-dimensional Riemannian manifold, 
$a > 0$ be a smooth positive function in $M$, and $(\omega_\alpha)_\alpha$ be a sequence in $(-\omega_a,\omega_a)$ such that 
$\omega_\alpha \to \omega$ as $\alpha \to +\infty$ for some $\omega \in [-\omega_a,\omega_a]$, where $\omega_a$ is as in \eqref{DefOm0a}. Also we let 
$\bigl((u_\alpha,v_\alpha)\bigr)_\alpha$ be a sequence 
of smooth positive solutions of \eqref{SWSyst} with phases $\omega_\alpha$ and $p = 6$. In particular, the $u_\alpha$'s and $v_\alpha$'s satisfy 
\eqref{SWSystAlphaCrit}. 
We assume that \eqref{ContrAssumptSubCptness} holds true. Following Druet and Hebey \cite{DruHebEinstLich}, see 
also Druet, Hebey and V\'etois \cite{DruHebVet}, 
there exists $C > 0$ such that for any $\alpha$ the following holds true. Namely that 
there exist $N_\alpha\in \mathbb{N}^\star$ and $N_\alpha$ critical points of $u_\alpha$, 
denoted by $\left(x_{1,\alpha}, x_{2,\alpha}, \dots, x_{N_\alpha,\alpha}\right)$, such that 
\begin{equation}\label{Eqt1Pr}
d_g\left(x_{i,\alpha},x_{j,\alpha}\right)^{\frac{1}{2}}u_\alpha(x_{i,\alpha}) \ge 1
\end{equation}
for all $i, j \in \left\{1,\dots,N_\alpha\right\}$, $i\neq j$, and 
\begin{equation}\label{Eqt2Pr}
\left(\min_{i=1,\dots,N_\alpha} d_g\left(x_{i,\alpha}, x\right)\right)^{\frac{1}{2}}u_\alpha(x) \le C
\end{equation}
for all $x \in M$ and all $\alpha$. We define $d_\alpha$ by
\begin{equation}\label{eqconcl1}
d_\alpha = \min_{1 \le i <   j \le N_\alpha} d_g\left(x_{i,\alpha},x_{j,\alpha}\right)\hskip.1cm.
\end{equation}
If $N_\alpha=1$, we set $d_\alpha= \frac{1}{4}i_g$, where $i_g$ is the injectivity radius of $(M,g)$. The first important lemma we prove in 
this section is that blow-up points are necessarily isolated in the sense that $d_\alpha \not\to 0$ as $\alpha \to +\infty$.

\begin{lem}\label{BLowUpPtsIsol} Let $(M,g)$ be a smooth compact Riemannian $3$-dimensional manifold and 
$\bigl((u_\alpha,v_\alpha)\bigr)_\alpha$ be a sequence 
of smooth positive solutions of \eqref{SWSystAlphaCrit} such that \eqref{ContrAssumptSubCptness} holds true. Then $d_\alpha \not\to 0$ as $\alpha \to +\infty$, where 
$d_\alpha$ is as in \eqref{eqconcl1}.
\end{lem}

\begin{proof}[Proof of Lemma \ref{BLowUpPtsIsol}] We proceed by contradiction and assume that $d_\alpha \to 0$ as $\alpha \to +\infty$. Then 
$N_\alpha \ge 2$ for $\alpha \gg 1$, and we can assume that the $x_{i,\alpha}$'s are such that 
$d_g(x_{1,\alpha},x_{i,\alpha}) \le d_g(x_{1,\alpha},x_{i+1,\alpha})$ 
for all $i = 2,\dots,N_\alpha$. Let $\delta \in (0,\frac{1}{2}i_g)$ be given. For $x \in B_0(\delta d_\alpha^{-1})$, we let
\begin{equation}\label{DefRendA}
\tilde u_\alpha(x) = d_\alpha^{1/2}u_\alpha\left(\exp_{x_{1,\alpha}}(d_\alpha x)\right)\hskip.1cm .
\end{equation}
We let also 
$\tilde h_\alpha(x) = h_\alpha\left(\exp_{x_{1,\alpha}}(d_\alpha x)\right)$, and $\tilde g_\alpha(x) = \left(\exp_{x_{1,\alpha}}^\star g\right)(d_\alpha x)$, 
where $h_\alpha$ is as in \eqref{Defhalpha}. Then, by \eqref{SWSystAlphaCrit},
\begin{equation}\label{EqtSec5}
\Delta_{\tilde g_\alpha}\tilde u_\alpha + d_\alpha^2\tilde h_\alpha\tilde u_\alpha = \tilde u_\alpha^5\hskip.1cm ,
\end{equation}
and we clearly have that $\tilde g_\alpha \to \xi$ in $C^2_{loc}(\mathbb{R}^3)$ as $\alpha \to +\infty$. Given $R > 0$ we let 
$1 \le N_{R,\alpha} \le N_\alpha$ be such that $d_g(x_{1,\alpha},x_{i,\alpha}) \le Rd_\alpha$ for all $1 \le i \le N_{R,\alpha}$, 
and $d_g(x_{1,\alpha},x_{i,\alpha}) > Rd_\alpha$ for all $N_{R,\alpha}+1 \le i \le N_\alpha$. We have that $N_{R,\alpha} \ge 2$ 
for all $R > 1$, and $(N_{R,\alpha})_\alpha$ is uniformly bounded for all $R > 0$. Mimicking 
the arguments in Druet and Hebey \cite{DruHebEinstLich}, 
see also Druet, Hebey and V\'etois \cite{DruHebVet}, given $R > 0$, there holds that
\begin{equation}\label{Dichot}
\begin{split}
&\hbox{either}\hskip.1cm \tilde u_\alpha(\tilde x_{i,\alpha}) = O(1)\hskip.1cm\hbox{for all}\hskip.1cm 1 \le i \le N_{R,\alpha}\hskip.1cm ,\\
&\hbox{or}\hskip.1cm \tilde u_\alpha(\tilde x_{i,\alpha}) \to +\infty\hskip.1cm\hbox{as}\hskip.1cm \alpha \to +\infty\hskip.1cm\hbox{for all}
\hskip.1cm 1 \le i \le N_{R,\alpha}\hskip.1cm ,
\end{split}
\end{equation}
where the $\tilde u_\alpha$'s are as in \eqref{DefRendA}, and 
\begin{equation}\label{defxia}
\tilde x_{i,\alpha} = \frac{1}{d_\alpha}\exp_{x_{1,\alpha}}^{-1}(x_{i,\alpha})\hskip.1cm .
\end{equation}
Now we split the proof into the study of two cases. In the first case we assume that 
there exist $R>0$ and $1\le i \le N_{R,\alpha}$ such that 
$\tilde u_\alpha(\tilde x_{i,\alpha}) = O(1)$. Then, by \eqref{Dichot}, $\tilde u_\alpha(\tilde x_{i,\alpha}) = O(1)$
for all $1\le i \le N_{R,\alpha}$ and all $R>0$. Noting that the two first equations in \eqref{hypCdts} are satisfied by 
$x_\alpha = x_{i,\alpha}$ and $\rho_\alpha = \frac{1}{8}d_\alpha$, it follows from \eqref{Limualpha} that the 
sequence $(\tilde u_\alpha)_\alpha$ is uniformly bounded in the balls $B_{\tilde x_{i,\alpha}}(1/2)$. Thus, by 
\eqref{EqtSec5} and elliptic theory, the sequence $(\tilde u_\alpha)_\alpha$ is bounded in $C^1_{loc}(\mathbb{R}^3)$. 
Up to a subsequence, still thanks to \eqref{EqtSec5}, we get that the $\tilde u_\alpha$'s converge in $C^1_{loc}(\mathbb{R}^3)$ 
as $\alpha \to +\infty$ to some $\tilde u$ which satisfies $\Delta\tilde u = \tilde u^5$ in $\mathbb{R}^3$. Moreover, 
$\tilde u$ has two critical points which are $0$ and the limit 
$\tilde x_2 \in S^2$ as $\alpha \to +\infty$ of the $\tilde x_{2,\alpha}$'s in \eqref{defxia}. By the classification result of 
Caffarelli, Gidas, and Spruck \cite{CafGidSpr}, this is impossible. In particular, we are left with the second case of our study, where 
we assume that there exist $R > 0$ and $1\le i \le N_{R,\alpha}$ such that 
$\tilde u_\alpha(\tilde x_{i,\alpha}) \to +\infty$ as $\alpha \to +\infty$. Then, by \eqref{Dichot}, 
$\tilde u_\alpha(\tilde x_{i,\alpha}) \to +\infty$ as $\alpha \to +\infty$
for all $1\le i \le N_{R,\alpha}$ and all $R>0$. The assumptions \eqref{hypCdts} are satisfied by 
$x_\alpha = x_{1,\alpha}$ and $\rho_\alpha = \frac{1}{8}d_\alpha$. Let $\tilde v_\alpha = \tilde u_\alpha(0)\tilde u_\alpha$. 
By \eqref{EqtSec5}, 
\begin{equation}\label{EqtSec5bis}
\Delta_{\tilde g_\alpha}\tilde v_\alpha + d_\alpha^2\tilde h_\alpha\tilde v_\alpha = \frac{1}{\tilde u_\alpha(0)^4}\tilde v_\alpha^5\hskip.1cm .
\end{equation}
Noting that $\tilde u_\alpha(0) \to +\infty$ as $\alpha \to +\infty$, mimicking again arguments from Druet and Hebey \cite{DruHebEinstLich}, 
and Druet, Hebey and V\'etois \cite{DruHebVet}, we get with \eqref{EqtSec5bis} that, up to a subsequence, 
$\tilde u_\alpha(0)\tilde u_\alpha \to \tilde G$ in $C^1_{loc}(\mathbb{R}^3\backslash\{\tilde x_i\}_{i\in I})$ as $\alpha \to +\infty$, where 
the $\tilde x_i$'s are the limits of the $\tilde x_{i,\alpha}$'s in \eqref{defxia}, and 
$I= \left\{1,\dots, \lim_{R\to +\infty}\lim_{\alpha\to +\infty} N_{R,\alpha}\right\}$.
Moreover, for any $R>0$, 
\begin{equation}\label{ExtGreenLim}
\begin{split}
\tilde G(x)
&= \sum_{i=1}^{\tilde{N}_R} \frac{\Lambda_i}{\vert x-\tilde{x}_i\vert} + \tilde{H}_R(x)\\
&= \frac{\Lambda_1}{\vert x\vert} + \left(\sum_{i=2}^{\tilde{N}_R} \frac{\Lambda_i}{\vert x-\tilde{x}_i\vert} + \tilde{H}_R(x)\right)
\end{split}
\end{equation}
in $B_0(R)$, where $\Lambda_i > 0$ for all $i$, $\tilde{H}_R$ is harmonic in $B_0(R)$, $2 \le \tilde{N}_R \le N_{2R}$ is such that 
$\vert \hat{x}_{\tilde{N}_R}\vert \le R$ and $\vert \hat{x}_{\tilde{N}_R +1}\vert > R$, 
and $N_{2R,\alpha} \to N_{2R}$ as $\alpha\to+\infty$. By Lemma \ref{LemSharpAsypts}, and \eqref{ExtGreenLim}, 
we get that $\Lambda_1 = \sqrt{3}$ and that 
\begin{equation}\label{ConclEqt1LemmIsol}
\sum_{i=2}^{\tilde{N}_R} \frac{\Lambda_i}{\vert\tilde x_i\vert} + \tilde{H}_R(0) = 0\hskip.1cm .
\end{equation}
Independently, by the maximum principle, since $\tilde G \ge 0$ and $\vert\tilde x_2\vert = 1$, there holds that 
\begin{equation}\label{ConclEqt2LemmIsol}
\sum_{i=2}^{\tilde{N}_R} \frac{\Lambda_i}{\vert\tilde x_i\vert} + \tilde{H}_R(0) \ge \Lambda_2 -\frac{\sqrt{3}}{R} - \frac{\Lambda_2}{(R-1)}\hskip.1cm .
\end{equation}
Choosing $R \gg 1$ sufficiently large, we get a contradiction by combining \eqref{ConclEqt1LemmIsol} 
and \eqref{ConclEqt2LemmIsol}. In particular, $d_\alpha \not\to 0$ as $\alpha \to +\infty$, and this proves Lemma \ref{BLowUpPtsIsol}. 
\end{proof}

Now that we know that blow-up points are isolated, we use elliptic theory and phase compensation to get strong convergence 
of the potential term in the nonlinear equation in \eqref{SWSystAlphaCrit}. Namely we prove that the following lemma holds true.

\begin{lem}\label{EllTheorPhaseComp} Let $(M,g)$ be a smooth compact Riemannian $3$-dimensional manifold and 
$\bigl((u_\alpha,v_\alpha)\bigr)_\alpha$ be a sequence 
of smooth positive solutions of \eqref{SWSystAlphaCrit} such that \eqref{ContrAssumptSubCptness} holds true. Then 
$(u_\alpha)_\alpha$ is bounded in $H^1$ and, up to a subsequence, $u_\alpha \rightharpoonup 0$ in $H^1$ and   
$v_\alpha \to v$ in $C^{0,\theta}$ as $\alpha \to +\infty$, where $v$ is a constant and $0 < \theta < 1$. Moreover, if 
$\lambda > 0$, then $v = 0$ and $h_\alpha \to a-\omega^2$ in $C^{0,\theta}$ as $\alpha \to +\infty$, where $0 < \theta < 1$, 
$\omega$ is the limit of the $\omega_\alpha$'s, and $h_\alpha$ is as in \eqref{Defhalpha}.
\end{lem}

\begin{proof}[Proof of Lemma \ref{EllTheorPhaseComp}] By Lemma \ref{BLowUpPtsIsol}, the sequence $(N_\alpha)_\alpha$ is uniformly bounded.
Up to a subsequence we may assume that $N_\alpha = N$ for all $\alpha$. Without loss of generality we may also assume that for any $\delta > 0$,
$$\sup_{B_{x_{i,\alpha}}(\delta)}u_\alpha \to +\infty$$
as $\alpha \to +\infty$ for all $i = 1,\dots,N$. It follows that there exists $\delta_0 > 0$, sufficiently small, such that 
\eqref{hypCdts} holds true with $x_\alpha = x_{i,\alpha}$ and $\rho_\alpha = \delta_0$ for all $i = 1,\dots,N$. We fix $i = 1,\dots,N$ arbitrary. 
By Lemma \ref{LemSharpAsypts}, $r_\alpha \not\to 0$ as $\alpha \to +\infty$. Then it follows from Lemma \ref{SharpEst} that there exist 
$r > 0$ and $C > 0$ such that
\begin{equation}\label{EstimateStrg}
u_\alpha(x) \le C\mu_\alpha^{\frac{1}{2}}d_g(x_\alpha,x)^{-1}
\end{equation}
for all $x \in B_{x_\alpha}(r)\backslash\{x_\alpha\}$ and all $\alpha$, where $\mu_\alpha$ is as in \eqref{DefMualpha}. 
In particular, together with \eqref{Limualpha}, this implies that
\begin{eqnarray*}
\int_{B_{x_\alpha}(r)}u_\alpha^6dv_g
&=& \int_{B_{x_\alpha}(\mu_\alpha)}u_\alpha^6dv_g 
+ \int_{B_{x_\alpha}(r)\backslash B_{x_\alpha}(\mu_\alpha)}u_\alpha^6dv_g\\
&\le& C
\end{eqnarray*}
for all $\alpha$, and since the $h_\alpha$'s in \eqref{Defhalpha} are bounded in $L^\infty$, we get with \eqref{SWSystAlphaCrit} that 
the $u_\alpha$'s are actually bounded in $H^1$. As a first consequence, since
\begin{equation}\label{EqtvAlpha}
\Delta_gv_\alpha+ \left(\lambda + q^2u_\alpha^2\right)v_\alpha = qu_\alpha^2\hskip.1cm ,
\end{equation}
and $0 \le v_\alpha \le \frac{1}{q}$, we get that that the $v_\alpha$'s are such that the $\left(\Delta_gv_\alpha + v_\alpha\right)$'s are 
bounded in $L^3$. By elliptic theory it follows the $v_\alpha$'s are bounded in $H^{2,3}$, and we can write that, up to a subsequence, 
\begin{equation}\label{ConvLInfty}
v_\alpha \to v
\end{equation}
in $C^{0,\theta}$ as $\alpha \to +\infty$ for some $v$, where $0 < \theta < 1$. Up to a subsequence, since the $u_\alpha$'s are bounded in $H^1$, 
we can assume that $u_\alpha \rightharpoonup u$ in $H^1$
as $\alpha \to +\infty$ for some $u \in H^1$. Let $x_i$ be the limit of the $x_{i,\alpha}$'s as $\alpha \to +\infty$, 
$i = 1,\dots,N$. By \eqref{hypCdts}, the $u_\alpha$'s are bounded in $L^\infty_{loc}(M\backslash S)$, where $S = \bigl\{x_1,\dots,x_N\bigr\}$. 
By elliptic estimates it follows that $u_\alpha \to u$ in $C^1_{loc}(M\backslash S)$ as $\alpha \to +\infty$. Coming back to 
\eqref{EstimateStrg} we then get that, necessarily, $u \equiv 0$ in $\bigcup_{i=1}^NB_{x_i}(r)$. By \eqref{ConvLInfty}, 
we can assume that $h_\alpha \to h$ in $C^{0,\theta}$ as $\alpha \to +\infty$ for some $h$, and $u$ solves
$\Delta_gu + hu = u^5$. 
In particular, the maximum principle applies and we actually have that $u \equiv 0$ in $M$. By Rellich-Kondrakov we can assume that 
$u_\alpha \to 0$ in $L^p$ as $\alpha \to +\infty$ for $p < 6$. Then, by \eqref{EqtvAlpha}, $\Delta_gv + \lambda v = 0$ in $M$. In 
particular, $v$ is a constant, and if $\lambda > 0$, then $v = 0$. This proves Lemma \ref{EllTheorPhaseComp}.
\end{proof}

In what follows we let $\delta > 0$ be given, sufficiently small, and let 
$\eta \in C^\infty(M\times M)$, $0 \le \eta \le 1$, be such that $\eta(x,y) = 1$ if $d_g(x,y) \le \delta$ 
and $\eta(x,y) = 0$ if $d_g(x,y) \ge 2\delta$. For $x\not= y$ we define
\begin{equation}\label{GreenPrincPart}
H(x,y) = \frac{\eta(x,y)}{\omega_2d_g(x,y)}\hskip.1cm ,
\end{equation}
where $\omega_2$ is the volume of the unit $2$-sphere. 
The following lemma, which will be used in the proof of the uniform bound in Theorem \ref{MainThm} when $p = 6$, 
establishes basic estimates for the Green's functions of Schr\"odinger's operators as well 
as a positive mass property for such operators that we deduce from the maximum principle and the 
positive mass theorem of Schoen and Yau \cite{SchYau1}.

\begin{lem}\label{PositiveMass} Let $(M,g)$ be a smooth compact Riemannian $3$-dimensional manifold and 
$\Lambda \in C^\infty(M)$ be such that $\Delta_g+\Lambda$ is coercive. The Green's function $G$ of $\Delta_g+\Lambda$ can be written as
\begin{equation}\label{GreenDec}
G(x,y) = H(x,y) + R(x,y)
\end{equation}
for all $(x,y) \in M\times M\backslash D$, where $D$ is the diagonal in $M\times M$, and $R$ is continuous 
in $M\times M$. Moreover, for any $x \in M$, there exists $C > 0$ such that
\begin{equation}\label{GreenFctProp1}
d_g(x,y)\vert\nabla R_x(y)\vert \le C
\end{equation}
for all $y \in M\backslash\{x\}$, where $R_x(y) = R(x,y)$, and there also holds that
\begin{equation}\label{GreenFctProp2}
\delta_\alpha\max_{y \in \partial B_x(\delta_\alpha)}\vert\nabla R_x(y)\vert = o(1)
\end{equation}
for all sequences $(\delta_\alpha)_\alpha$ of positive real numbers converging to zero. At last, if 
we assume that $\Lambda \le \frac{1}{8}S_g$, the inequality being strict at least at one point if $(M,g)$ 
is conformally diffeomorphic to the unit $3$-sphere, then $R(x,x) > 0$ for all $x \in M$.
\end{lem}

\begin{proof}[Proof of Lemma \ref{PositiveMass}] The decomposition \eqref{GreenDec} is well  known. It is also known that 
there exists $C > 0$ such that $d_g(x,y)\vert\Delta_gR_x(y)\vert \le C$ for all $y \in M\backslash\{x\}$. Possible references 
for such properties are Aubin \cite{Aub1}, or Druet, Hebey and Robert \cite{DruHebRob}. We refer also to Robert \cite{RobGreen}. Now we 
establish \eqref{GreenFctProp1} and \eqref{GreenFctProp2}. We let $(y_\alpha)_\alpha$ be an arbitrary sequence in $M\backslash\{x\}$ such that 
$y_\alpha \to x$ as $\alpha \to +\infty$. Let $\delta_\alpha = d_g(x,y_\alpha)$ and $R_\alpha(y) = R_x\left(\exp_x(\delta_\alpha y)\right)$ 
for $y \in \mathbb{R}^3$. Let also $g_\alpha$ be the metric given by $g_\alpha(y) = (\exp_x^\star g)(\delta_\alpha y)$, and 
$\tilde y_\alpha \in \mathbb{R}^3$ be such that $y_\alpha = \exp_x(\delta_\alpha\tilde y_\alpha)$. There holds that 
$\vert\Delta_g R_\alpha(y)\vert \le C\delta_\alpha \vert y\vert^{-1}$, that $(R_\alpha)_\alpha$ is bounded in $L^\infty$, and that 
$g_\alpha \to \xi$ in $C^1_{loc}(\mathbb{R}^3)$ as $\alpha \to +\infty$, where $\xi$ is the Euclidean metric. There also holds that 
$\vert\tilde y_\alpha\vert = 1$ for all $\alpha$. Let $\tilde y$ be such that $\tilde y_\alpha \to \tilde y$ as $\alpha \to +\infty$. 
Since $\vert\tilde y\vert = 1$ it follows from standard elliptic theory that $(R_\alpha)_\alpha$ is bounded in the 
$C^1$-topology in the Euclidean ball of center $\tilde y$ and radius $1/4$. Since $(y_\alpha)_\alpha$ is arbitrary, this proves 
\eqref{GreenFctProp1}. Noting that $\Delta_{g_\alpha}R_\alpha \to 0$ uniformly in compact subsets of 
$\mathbb{R}^3\backslash\{0\}$ as $\alpha \to +\infty$, we get that $R_\alpha \to R$ in $C^1_{loc}(\mathbb{R}^3)$ as 
$\alpha \to +\infty$, where $R$ is harmonic and bounded in $\mathbb{R}^3\backslash\{0\}$. By Liouville's theorem, $R$ is constant. 
This implies \eqref{GreenFctProp2}. Now it remains to prove the positive mass property that $R(x,x) > 0$ for all $x$ if 
we assume that $\Lambda \le \frac{1}{8}S_g$, the inequality being strict at least at one point if $(M,g)$ 
is conformally diffeomorphic to the unit $3$-sphere. Let $\tilde G$ be the Green's function of the conformal Laplacian 
$\Delta_g+\frac{1}{8}S_g$. Let $x \in M$ and $h \ge 0$ be smooth and such that
$$h \le \left(\frac{1}{8}S_g-\Lambda\right)\tilde G_x\hskip.1cm .$$
If $\Lambda \equiv \frac{1}{8}S_g$, then $h \equiv 0$, but if not the case we can take $h \not\equiv 0$. Let $\tilde h$, smooth, 
be such that $\Delta_g\tilde h + \Lambda\tilde h = h$. 
Then $\tilde h \ge 0$ and $\tilde h \not\equiv 0$ if $h\not\equiv 0$. In particular, by the maximum principle, 
$\tilde h > 0$ in $M$ if $\tilde h \not\equiv 0$. Let $\mathcal{H} = G_x - \tilde G_x - \tilde h$. Noting that
$$\Delta_g\mathcal{H} + \Lambda\mathcal{H} \ge 0$$
in $M$, and that by the local expansions of $G_x$ and $\tilde G_x$, 
$\mathcal{H}$ is continuous in $M$, we get from the maximum principle that $\mathcal{H} \ge 0$ in $M$. In particular, by \eqref{GreenSchYau}, 
$R(x,x) \ge A + \tilde h(x)$, and we get that $R(x,x) > 0$ by 
the positive mass theorem of Schoen and Yau \cite{SchYau1}. This ends the proof of the lemma.
\end{proof}

Thanks to Lemma \ref{EllTheorPhaseComp} and Lemma \ref{PositiveMass} we can now prove the uniform bound in Theorem \ref{MainThm} when $p = 6$. 
In the process we use the asymptotic control we obtained in Lemma \ref{SharpEst}.

\begin{proof}[Proof of the uniform bound in Theorem \ref{MainThm} when $p = 6$] In what follows we consider a smooth compact $3$-dimensional Riemannian manifold $(M,g)$, 
and let $a > 0$ be a smooth positive function in $M$, $\omega \in (-\omega_a,\omega_a)$, and $(\omega_\alpha)_\alpha$ be a sequence such that 
$\omega_\alpha \to \tilde\omega$ as $\alpha \to +\infty$ for some $\tilde\omega \in [-\omega_a,-\omega]\bigcup[\omega,\omega_a]$, 
where $\omega_a$ is as in \eqref{DefOm0a}. We assume either that $\vert\tilde\omega\vert < \omega_a$, or that $\Delta_g + (a-\omega_a^2)$ 
is a coercive operator. In case $a$ is constant and $\tilde\omega = \omega_a$, we apply the arguments in Section \ref{PrfThm2}, noting that $v_\alpha = 1/q$ for all 
$\alpha$ in case $\lambda = 0$. We let 
$\bigl((u_\alpha,v_\alpha)\bigr)_\alpha$ be a sequence 
of smooth positive solutions of \eqref{SWSyst} with phases $\omega_\alpha$ and $p = 6$. In particular, the $u_\alpha$'s and $v_\alpha$'s satisfy 
\eqref{SWSystAlphaCrit}. 
We assume by contradiction that \eqref{ContrAssumptSubCptness} holds true and we assume \eqref{MainAssumpt}, 
with the property that \eqref{MainAssumpt} is strict at least at one point if $(M,g)$ is conformally diffeomorphic to the unit $3$-sphere 
and $\omega\lambda = 0$. By Lemma \ref{BLowUpPtsIsol}, the sequence $(N_\alpha)_\alpha$ is uniformly bounded. 
Up to a subsequence we may assume that $N_\alpha = N$ for all $\alpha$. We let $x_i$ be the limit of the $x_{i,\alpha}$'s as $\alpha \to +\infty$, 
and let the $\mu_{i,\alpha}$ be as in \eqref{DefMualpha} given by
$$\mu_{i,\alpha} = u_\alpha(x_{i,\alpha})^{-2}$$
for all $i = 1,\dots,N$ and all $\alpha$. Without loss of generality 
we can assume that $\mu_{i,\alpha} \to 0$ as $\alpha \to+\infty$ for all $i$. We reorganize the $i$'s such that, up to a subsequence, 
$$\mu_{1,\alpha} = \max_i\mu_{i,\alpha}$$
and we define $\mu_i \ge 0$ by
\begin{equation}\label{DefMui}
\mu_i = \lim_{\alpha \to+\infty}\frac{\mu_{i,\alpha}}{\mu_{1,\alpha}}\hskip.1cm .
\end{equation}
By Lemma \ref{SharpEst} and the Harnack inequality for any $\delta > 0$ there exists $C > 0$ such that 
\begin{equation}\label{EstEndPrf}
u_\alpha \le C\mu_{1,\alpha}^{1/2}
\end{equation}
in $M\backslash\bigcup_{i=1}^NB_{x_{i,\alpha}}(\delta)$ for all $\alpha$. There holds that
\begin{equation}\label{EqtRenU}
\Delta_g(\mu_{1,\alpha}^{-1/2}u_\alpha) + h_\alpha(\mu_{1,\alpha}^{-1/2}u_\alpha) 
= \mu_{1,\alpha}^2(\mu_{1,\alpha}^{-1/2}u_\alpha)^5
\end{equation}
for all $\alpha$, where $h_\alpha$ is as in \eqref{Defhalpha}. By Lemma \ref{EllTheorPhaseComp}, the $h_\alpha$'s converge in $C^{0,\theta}$, 
$\theta \in (0,1)$. Let $h$ be the limit of the $h_\alpha$'s. Still by Lemma \ref{EllTheorPhaseComp}, $h = a-\tilde\omega^2$ if $\lambda > 0$. 
In general, $h = a - \tilde\omega^2(qv-1)^2$ so that $h \ge a-\tilde\omega^2 \ge a-\omega_a^2$. By our assumptions that either $\vert\tilde\omega\vert < \omega_a$, or 
$\Delta_g + (a-\omega_a^2)$ is coercive, we get that $\Delta_g + h$ is coercive. 
Combining \eqref{EstEndPrf} and \eqref{EqtRenU}, we get thanks to standard elliptic theory that
\begin{equation}\label{ConvToGr}
\mu_{1,\alpha}^{-1/2}u_\alpha \to \mathcal{H}
\end{equation} 
in $C^1_{loc}(M\backslash S)$ as $\alpha \to +\infty$, where $S = \left\{x_1,\dots,x_N\right\}$. By \eqref{SWSystAlphaCrit}, Lemma \ref{SharpEst}, \eqref{EstEndPrf}
and the convergence of the $h_\alpha$'s to $h$, we can write that
\begin{equation}\label{EqtForH}
\Delta_g\mathcal{H} + h\mathcal{H} = \sqrt{3}\omega_2\sum_{i=1}^N\mu_i^{1/2}\delta_{x_i}\hskip.1cm .
\end{equation}
Since $\Delta_g + h$ is coercive, we get from \eqref{EqtForH} that
\begin{equation}\label{ExplFormForH}
\mathcal{H}(x) = \sqrt{3}\omega_2\sum_{i=1}^N\mu_i^{1/2}\left(H(x_i,x) + R(x_i,x)\right)
\end{equation}
where $H$ and $R$ are as in Lemma \ref{PositiveMass} with $\Lambda = h$. Let $i = 1,\dots,N$ be arbitrary and $X_\alpha$ be the $1$-form given by 
$X_\alpha = \nabla f_\alpha$, where $f_\alpha(x) = \frac{1}{2}d_g(x_{i,\alpha},x)^2$. We apply the 
Pohozaev identity in Druet and Hebey \cite{DruHebStrongCoup1} to $u_\alpha$ in $B_{x_{i,\alpha}}(r)$, $r > 0$ small. By Lemma \ref{SharpEst}, multiplying 
the Pohozaev identity by $\mu_{1,\alpha}^{-1}$, letting $\alpha \to +\infty$ tend to infinity, we get that
\begin{equation}\label{PohId}
\begin{split}
&\frac{1}{6}\int_{\partial B_{x_i}(r)}\left(\hbox{div}_gX\right)\mathcal{H}\partial_\nu\mathcal{H} d\sigma_g\\
&- \int_{\partial B_{x_i}(r)}\left(\frac{1}{2}X(\nu)\vert\nabla\mathcal{H}\vert^2 - (X,\nabla\mathcal{H})\partial_\nu\mathcal{H}\right)d\sigma_g\\
&= \frac{1}{12}\int_{\partial B_{x_i}(r)}\left(\partial_\nu(\hbox{div}_gX)\right)\mathcal{H}^2d\sigma_g + o(1)\hskip.1cm ,
\end{split}
\end{equation}
where $X = \nabla f$, $f = \frac{1}{2}d_g(x_i,\cdot)^2$, $\nu$ is the unit outward normal derivative to $\partial B_{x_i}(r)$, and $o(1) \to 0$ 
as $r \to 0$. We have that $\hbox{div}_gX = 3 + O\left(d_g(x_i,x)^2\right)$ and that $\left\vert\nabla\hbox{div}_gX\right\vert = O\left(d_g(x_i,x)\right)$. 
By Lemma \ref{SharpEst} there also holds that $\vert\mathcal{H}\vert \le Cd_g(x_i,\cdot)^{-1}$ in $M\backslash\{x_i\}$. Thanks to 
\eqref{GreenFctProp1} we then get that
\begin{equation}\label{PohN3Comp2}
\lim_{r\to 0}\int_{\partial B_{x_i}(r)} \left(\partial_\nu \left(div_g X\right)\right)\mathcal{H}^2 d\sigma_g = 0
\end{equation}
and that
\begin{equation}\label{PohN3Comp3}
\frac{1}{6}\int_{\partial B_{x_i}(r)}\left(\hbox{div}_gX\right)\mathcal{H}\partial_\nu\mathcal{H} d\sigma_g 
= \frac{1}{2} \int_{\partial B_{x_i}(r)}\mathcal{H}\partial_\nu\mathcal{H} d\sigma_g + o(1)
\end{equation}
as $r \to 0$. Choosing $\delta > 0$ in the definition of $\eta$ in \eqref{GreenPrincPart} such that $d_g(x_j,x_k) \ge 4\delta$ for all $j, k = 1,\dots,N$ 
such that $x_j\not= x_k$, we get that $\mathcal{R}(x_j,x_i) \ge 0$ for all $j \not= i$. By \eqref{GreenFctProp2} and \eqref{ExplFormForH}, we compute
\begin{equation}\label{PassLimPohN3}
\begin{split}
&\frac{1}{2} \int_{\partial B_{x_i}(r)}\mathcal{H}\partial_\nu\mathcal{H} d\sigma_g  
- \int_{\partial B_{x_i}(r)}\left(\frac{1}{2}X(\nu)\vert\nabla\mathcal{H}\vert^2 - (X,\nabla\mathcal{H})\partial_\nu\mathcal{H}\right)d\sigma_g\\
&= -\frac{3\omega_2}{2}\mu_i^{1/2}\sum_{j=1}^N\mu_j^{1/2}\mathcal{R}(x_j,x_i) + o(1)\hskip.1cm .
\end{split}
\end{equation}
Combining \eqref{PohId}--\eqref{PassLimPohN3}, letting $r \to 0$, it follows that
\begin{equation}\label{MassConcl}
\mu_i^{1/2}\mu_j^{1/2}\mathcal{R}(x_j,x_i) = 0
\end{equation}
for all $i, j$. Letting $i, j = 1$, we have that $\mu_1 = 1$, and it follows from \eqref{MassConcl} that 
$\mathcal{R}(x_1,x_1) = 0$. By assumption, $h < \frac{1}{8}S_g$ if $\omega\lambda \not= 0$, and in case 
$\omega\lambda = 0$, we get that $h \le a \le \frac{1}{8}S_g$ with the property that either the manifold is 
not conformally diffeomorphic to the unit $3$-sphere, or that the inequality is trict at one point. By Lemma 
\ref{PositiveMass} it follows that $\mathcal{R}(x_1,x_1) > 0$ and we get a contradiction. 
This proves that there exists $C > 0$ such that $\Vert u_\alpha\Vert_{L^\infty} \le C$ for all $\alpha$ and all 
$\omega_\alpha \in K(\omega)$. Since we also have that 
$0 \le v_\alpha \le \frac{1}{q}$ for all $\alpha$, it follows from elliptic theory and \eqref{SWSystAlphaCrit} that the $u_\alpha$'s and $v_\alpha$'s are 
bounded in $H^{2,p}$ for all $p > 1$. The $C^{2,\theta}$-bound easily follows. This ends the proof of Theorem \ref{MainThm}.
\end{proof}

\section{Proof of Theorem \ref{SecThm}}\label{PrfThm2}

We use here part of the analysis developed in Sections \ref{blowupest} and \ref{Sect3}, together with a nice concluding 
argument from Br\'ezis and Li \cite{BreLi}. 
As in the proof of Theorem \ref{MainThm} we proceed by contradiction. We let $\bigl((u_\alpha,v_\alpha)\bigr)_\alpha$ be a sequence 
of smooth positive solutions of 
\begin{equation}\label{SWSystAlphaPotentCrit}
\begin{cases}
\Delta_gu_\alpha + a_\alpha u_\alpha = u_\alpha^5 + \omega_\alpha^2\left(qv_\alpha-1\right)^2u_\alpha\\
\Delta_gv_\alpha + \left(\lambda+q^2u_\alpha^2\right)v_\alpha = qu_\alpha^2
\end{cases}
\end{equation}
such that $\max_Mu_\alpha \to +\infty$ as $\alpha \to +\infty$, 
where the $a_\alpha$'s are smooth positive functions and the $\omega_\alpha$'s are phases in $\left(-\omega_{a_\alpha},\omega_{a_\alpha}\right)$ such that 

(i) either $a_\alpha-\omega_\alpha^2 \to 0$ in $L^\infty(M)$ as $\alpha \to +\infty$ and $\lambda > 0$\hskip.1cm ,\hskip.1cm or
 
 (ii) $a_\alpha \to 0$ in $L^\infty(M)$ and $\omega_\alpha \to 0$ as $\alpha \to +\infty$, and $\lambda \ge 0$\hskip.1cm .

\noindent The estimates in Section \ref{blowupest} as well as Lemma \ref{BLowUpPtsIsol}, which establishes that $d_\alpha \not\to 0$ as $\alpha \to +\infty$, 
still hold true in the present situation. Let $h_\alpha$ be given by
$$h_\alpha = a_\alpha - \omega_\alpha^2\left(qv_\alpha-1\right)^2\hskip.1cm .$$
Without loss of generality, up to passing to a subsequence, we may assume that $a_\alpha \to a$ in $C^{0,\theta}$ and that $\omega_\alpha \to \omega$ as $\alpha \to +\infty$. 
Then, by the arguments developed in Lemma \ref{EllTheorPhaseComp}, we get that $v_\alpha \to v$ in $C^{0,\theta}$ as $\alpha \to +\infty$, and we then get with (i) and (ii) that 
$h_\alpha \to 0$ in $C^{0,\theta}$ as $\alpha \to +\infty$. As in the proof of Theorem \ref{MainThm} in Section \ref{Sect3}, the convergence $\mu_{1,\alpha}^{-1/2}u_\alpha \to \mathcal{H}$
in \eqref{ConvToGr} holds true. However, in the present situation, $h \equiv 0$ and we get that
$\Delta_g\mathcal{H} = 0$ in $M\backslash S$, while
$$\Delta_g\mathcal{H} = \sqrt{3}\omega_2\sum_{i=1}^N\mu_i^{1/2}\delta_{x_i}$$
in the sense of distributions, where we adopt the notations of Section \ref{Sect3}. In other words, $\mathcal{H}$ is a nonnegative harmonic function with poles, 
and this is impossible. Indeed, from now on, without loss of generality, we may assume that $\mu_i > 0$ 
for all $i$ (we know that at least $\mu_1 > 0$). Then $\mathcal{H} \ge 0$ and $\mathcal{H}$ is not constant. 
Let $G$ be a Green's function of $\Delta_g$, and $G_i = G(x_i,\cdot)$.
By regularity theory,
$$\mathcal{H} = \sqrt{3}\omega_2\sum_{i=1}^N\mu_i^{1/2}G_i + \mathcal{F}\hskip.1cm,$$
where $\mathcal{F}$ is smooth. In particular, by standard expansion of the Green's function at its pole, see, for instance, Aubin 
\cite{Aub2}, there exists $x \in M\backslash S$ where $\mathcal{H}$ 
attains its minimum. 
By the maximum principle we would get that $\mathcal{H}$ is actually constant in $M\backslash S$, a contradiction. 

\medskip\noindent Acknowledgement: Electronic version of an article published as [CCM, 12, 5, 2010, 831-869][DOI No: 10.1142/S0219199710004007]\copyright[copyright World Scientific Publishing Company][http://www.worldscinet.com/ccm/ccm.shtml]


\begin{thebibliography}{72}

\bibitem{AmbRab} Ambrosetti, A., and Rabinowitz, P.H., 
Dual variational methods in critical point theory and applications.
{\it J. Functional Analysis}, 14, 349--381, 1973.

\bibitem{AmbRui} Ambrosetti, A., and Ruiz, D., Multiple bound states for the Schr\"odinger-Poisson problem, 
{\it Commun. Contemp. Math.}, 10, 391--404, 2008.

\bibitem{Aub1} Aubin, T., Equations diff\'erentielles non lin\'eaires et probl\`eme de Yamabe concernant la courbure scalaire,
{\it J. Math. Pures Appl.}, 55, 269-296, 1976.

\bibitem{Aub2} \bysame,  {\it Nonlinear analysis on manifolds. Monge-Amp\`ere equations}, 
Grundlehren der Mathematischen Wissenschaften, 252, Springer,
New York--Berlin, 1982.

\bibitem{AprMug1} D'Aprile, T., and Mugnai, D., Solitary waves for nonlinear Klein-Gordon-Maxwell and 
Schršdinger-Maxwell equations, {\it Proc. Roy. Soc. Edinburgh Sect. A}, 134, 893--906, 2004. 

\bibitem{AprMug2} \bysame, Non-existence results for the coupled Klein-Gordon-Maxwell equations, 
{\it Adv. Nonlinear Stud.} 4 (2004), 307--322, 2004.

\bibitem{AprWei1} D'Aprile, T., and Wei, J., Layered solutions for a semilinear elliptic system in a ball, 
{\it J. Differential Equations}, 226, 269--294, 2006. 

\bibitem{AprWei2} \bysame, Clustered solutions around harmonic centers to a coupled elliptic system, 
{\it Ann. Inst. H. PoincarŽ Anal. Non LinŽaire}, 24, 605--628, 2007.

\bibitem{AvePis} D'Avenia, P., and Pisani, L., Nonlinear Klein-Gordon equations coupled with Born-Infeld type equations, 
{\it Electron. J. Differential Equations}, 26, 1--13, 2002.

\bibitem{AvePisSic} D'Avenia, P., Pisani, L., and Siciliano, G., Klein-Gordon-Maxwell system in a 
bounded domain, {\it Preprint}, 2008.

\bibitem{AvePisSic2} \bysame, Dirichlet and Neumann problems
for Klein-Gordon-Maxwell systems, {\it Preprint}, 2008.

\bibitem{AzzAvePom} Azzollini, A., D'Avenia, P., and Pomponio, A., On the 
Schr\"odinger-Maxwell equations under the effect of a general nonlinear term, 
{\it Preprint}, 2009.

\bibitem{AzzPom1} Azzollini, A., and Pomponio, A., Ground state solutions for the nonlinear Schr\"odinger-Maxwell equations, 
{\it J. Math. Anal. Appl.},  345, 90--108, 2008.

\bibitem{AzzPom2} \bysame, Ground state solutions for the nonlinear Klein-Gordon-Maxwell equations, {\it Preprint}, 2008.

\bibitem{BecMauSel} Bechouche, P., Mauser, N.J., and Selberg, S., 
Nonrelativistic limit of Klein-Gordon-Maxwell to Schr\"odinger-Poisson, 
{\it Amer. J. Math.}, 126, 31--64, 2004.

\bibitem{BenFor0} Benci, V., and Fortunato, D., Solitary waves of the nonlinear Klein-Gordon field equation coupled with the 
Maxwell equations, {\it Rev. Math. Phys.}, 14, 409--420, 2002.

\bibitem{BenFor1} \bysame, Solitary Waves in Abelian gauge theories, {\it Adv. Nonlinear Stud.}, 8, 327--352, 2008.

\bibitem{BenFor2} \bysame, Existence of hylomorphic solitary waves in Klein-Gordon and in 
Klein-Gordon-Maxwell equations, {\it Preprint}, 2009.

\bibitem{BenFor3} \bysame, Hylomorphic vortices in Abelian gauge theories, {\it Preprint}, 2009.

\bibitem{BerMal} Berti, M., and Malchiodi, A., Non-compactness and multiplicity results 
for the Yamabe problem on $S^n$, {\it J. Funct. Anal.}, 180, 210--241, 2001.

\bibitem{Bon} Bonanno, C., Multiplicity of positive solutions for nonlinear field equations in 
$\mathbb{R}^n$, {\it Preprint}, 2009.

\bibitem{Bre} Brendle, S., Blow-up phenomena for the Yamabe equation, {\it J. Amer. Math. Soc.},  21, 
951--979, 2008.

\bibitem{BreSur} \bysame, On the conformal scalar curvature equation and related problems, {\it Surveys on Differential 
Geometry}, Surveys in differential geometry. Vol. XII,  1--19, 2008.

\bibitem{BreMar} Brendle, S., and Marques, Blow-up phenomena for the Yamabe equation II, {\it J. Differential Geom.}, 
81, 225-250, 2009.

\bibitem{BreLi} Br\'ezis, H., and Li, Y.Y., Some nonlinear elliptic equations have only constant solutions, 
{\it J. Partial Differential Equations}, 19, 208--217, 2006.

\bibitem{CafGidSpr} Caffarelli, L. A., Gidas, B., and Spruck, J., 
Asymptotic symmetry and local 
behavior of semilinear elliptic equations with critical Sobolev growth, 
{\it Comm. Pure Appl. Math.}, 42, 271--297, 1989.

\bibitem{Cas} Cassani, D., Existence and non-existence of solitary waves for the critical Klein-Gordon equation coupled with Maxwell's equations, 
{\it Nonlinear Anal.}, 58, 733--747, 2004.

\bibitem{ChoBru} Choquet-Bruhat, Y., 
Solution globale des \'equations de Maxwell-Dirac-Klein-Gordon,
{\it Rend. Circ. Mat. Palermo}, 31, 267--288, 1982. 

\bibitem{Deu} Deumens, E.,
The Klein-Gordon-Maxwell nonlinear system of equations, 
{\it Physica D.}, 18, 371--373, 1986.

\bibitem{Dru1} Druet, O., From one bubble to several bubbles: The low-dimensional case, 
{\it J. Differential Geom.}, 63, 399--473, 2003.

\bibitem{Dru2} \bysame , Compactness for Yamabe metrics in low dimensions, 
{\it Internat. Math. Res. Notices}, 23, 1143--1191, 2004.

\bibitem{DruHebTrans} Druet, O., and Hebey, E. , Blow-up examples for second order elliptic PDEs 
of critical Sobolev growth, {\it Trans. Amer. Math. Soc.}, 357, 1915--1929, 2004.

\bibitem{DruHebIMRS} \bysame, Elliptic equations of Yamabe type, {\it International 
Mathematics Research Surveys}, 1, 1--113, 2005.

\bibitem{DruHebEinstLich} \bysame, Stability and instability for Einstein-scalar field 
Lichnerowicz equations on compact Riemannian manifolds, {\it Math. Z.}, to appear.

\bibitem{DruHebStrongCoup1} \bysame, Stability for strongly coupled critical elliptic systems in a fully 
inhomogeneous medium, {\it Analysis and PDEs}, to appear.

\bibitem{DruHebRob} Druet, O., Hebey, E., and Robert, F.,  
{\it Blow-up theory for elliptic PDEs in Riemannian geometry}, Mathematical Notes, 
Princeton University Press, vol. 45, 2004.

\bibitem{DruHebVet} Druet, O., Hebey, E., and V\'etois, J., Bounded stability for strongly coupled 
critical elliptic systems below the geometric threshold of the conformal Laplacian, {\it J. Funct. Anal.}, 
to appear. 

\bibitem{DruLau} Druet, O., and Laurain, P., Stability of the Pohozaev obstruction in dimension 3, 
{\it J. Eur. Math. Soc.}, to appear.

\bibitem{EarMon} Eardley, D., and Moncrief, V., 
The global existence of Yang-Mills-Higgs fields in 4-dimensional Minkowski space. I. Local existence and smoothness, 
{\it Comm. Math. Phys.}, 83, 171--191, 1982. 

\bibitem{GidSpr} Gidas, B., and Spruck, J., A priori bounds for positive solutions of nonlinear elliptic equations, 
{\it Comm. Part. Diff. Eq.}, 6, 883--901, 1981.

\bibitem{CIMSBook} Hebey, E., {\it Nonlinear analysis on manifolds: Sobolev spaces and inequalities}, 
CIMS Lecture Notes, Courant Institute of Mathematical Sciences, Vol. 5, 1999. 
Second edition published by the American Mathematical Society, 2000.

\bibitem{HebTru} Hebey, E., and Truong, T.T., Static Klein-Gordon-Maxsell-Proca systems in $4$-dimensional closed 
manifolds, {\it J. Reine Angew. Math.}, to appear.

\bibitem{HebVau1} Hebey, E., and Vaugon, M., The best constant problem in the Sobolev 
embedding theorem for complete Riemannian manifolds, 
{\it Duke Math. J.}, 79, 1995, 235--279.

\bibitem{HebVau2} \bysame , Meilleures constantes dans le th\'eor\`eme 
d'inclusion de Sobolev, {\it Ann. Inst. H.  Poincar\'e. Anal. Non Lin\'eaire}, 
13, 1996, 57--93.

\bibitem{HebVauSym} \bysame, Sobolev spaces in the presence of symmetries, 
{\it J. Math. Pures Appl.}, 76, 1997, 859--881.

\bibitem{IanVai} Ianni, I., and Vaira, G., On concentration of positive bound states for the Schr\"odinger-Poisson 
problem with potentials, {\it Adv. Nonlinear Stud.}, 8, 573--595, 2008.

\bibitem{KhuMarSch} Khuri, M., Marques, F. C., and Schoen, R., A compactness theorem for the Yamabe Problem, 
{\it J. Differential Geom.}, to appear.

\bibitem{KlaMac} Klainerman, S., and Machedon, M., 
On the Maxwell-Klein-Gordon equation with finite energy, 
{\it Duke Math. J.}, 74, 19--44, 1994. 

\bibitem{LiZhaMS} Li, Y.Y., and Zhang, L., Liouville-type theorems and Harnack-type inequalities 
for semilinear elliptic equations, {\it J. Anal. Math.}, 90, 27--87, 2003.

\bibitem{LiZha} \bysame, A Harnack type inequality for the Yamabe 
equations in low dimensions, {\it Calc. Var. Partial Differential Equations}, 20, 133--151, 2004.

\bibitem{LiZha2} \bysame, Compactness of solutions to the Yamabe problem II, {\it Calc. Var. PDE}, 
24, 185--237, 2005.

\bibitem{LiZhuMS} Li, Y.Y., and Zhu, M., Uniqueness theorems through the method of moving spheres, {\it Duke Math. J.}, 
80, 383--417, 1995.

\bibitem{LiZhu} \bysame, Yamabe type equations on three dimensional 
Riemannian manifolds, {\it Commun. Contemp. Math.}, 1, 1--50, 1999.

\bibitem{Lon} Long, E., Existence and stability of solitary waves in non-linear Klein-Gordon-Maxwell equations, 
{\it Rev. Math. Phys.}, 18, 747--779, 2006.

\bibitem{MacSte} Machedon, M., and Sterbenz, J., 
Almost optimal local well-posedness for the 
(3+1)-dimensional Maxwell-Klein-Gordon equations, 
{\it J. Amer. Math. Soc.} 17, 297--359, 2004.

\bibitem{Mar} Marques, F.C., A priori estimates for the Yamabe problem in the non-locally 
conformally flat case, {\it J. Differential Geom.}, 71, 315--346, 2005.

\bibitem{MasNak1} Masmoudi, N., and Nakanishi, K., 
Uniqueness of finite energy solutions for Maxwell-Dirac and Maxwell-Klein-Gordon equations, 
{\it Comm. Math. Phys.}, 243, 123--136, 2003. 

\bibitem{MasNak2} \bysame, 
Nonrelativistic limit from Maxwell-Klein-Gordon and Maxwell-Dirac to Poisson-Schršdinger, 
{\it Int. Math. Res. Not.}, 13, 697--734, 2003.

\bibitem{Mug} Mugnai, D., Coupled Klein-Gordon and Born-Infeld-type equations: looking for solitary waves, 
{\it Proc. R. Soc. Lond. Ser. A}, Math. Phys. Eng. Sci., 460, 1519--1527, 2004.

\bibitem{Pet} Petrescu, D.M., 
Time decay of solutions of coupled Maxwell-Klein-Gordon equations, 
{\it Commun. Math. Phys.}, 179, 11--24, 1996. 

\bibitem{RobGreen} Robert, F., Green's Functions estimates for elliptic type operators, {\it Preprint}, 2006.

\bibitem{RodTao} Rodnianski, I., and Tao, T., 
Global regularity for the Maxwell-Klein-Gordon equation with small 
critical Sobolev norm in high dimensions, 
{\it Comm. Math. Phys.}, 251, 377--426, 2004. 

\bibitem{Rui} Ruiz, D., Semiclassical states for coupled Schr\"odinger-Maxwell equations: concentration around a sphere, 
{\it Math. Models Methods Appl. Sci.}, 15, 141--164, 2005.

\bibitem{Sch1} Schoen, R., Conformal deformation of a Riemannian metric to constant scalar curvature, 
{\it J. Differential Geom.}, 20, 479--495, 1984.

\bibitem{SchNotes} \bysame, {\it Lecture notes from courses at Stanford}, written by D.Pollack, preprint, 1988.

\bibitem{Sch3} \bysame , Variational theory for the total scalar curvature functional for Riemannian 
metrics and related topics, in {\it Topics in Calculus of Variations} (Montecatini Terme, 1987), Lecture Notes 
in Math., vol. 1365, Springer-Verlag, Berlin, 120--154, 1989.

\bibitem{Sch4} \bysame , On the number of constant scalar curvature metrics in a conformal class, 
in {\it Differential Geometry: A Symposium in Honor of Manfredo do Carmo}, Proc. Int. Conf. (Rio de Janeiro, 1988). 
Pitman Monogr. Surveys Pure Appl. Math., vol. 52, Longman Sci. Tech., Harlow, 311--320, 1991.

\bibitem{SchYau1} Schoen, R.M., and Yau, S.T., On the proof of the positive mass conjecture in general 
relativity, {\it Comm. Math. Phys.}, 65, 45--76, 1979.

\bibitem{SchYau2} \bysame , Proof of the positive action conjecture in quantum relativity, {\it Phys. Rev. Let.}, 
42, 547--548, 1979.

\bibitem{SchYau3} \bysame , Conformally flat manifolds, Kleinian groups and scalar curvature, 
{\it Invent. Math.}, 92, 47--71, 1988.

\bibitem{Tao1} Tao, T., Global behaviour of nonlinear dispersive and wave equations, 
{\it Current Developments in Mathematics}, Vol. 2006 (2008), 255--340.

\bibitem{Tru} Trudinger, N.S., Remarks concerning the conformal deformation of Riemannian 
structures on compact manifolds, {\it Ann. Scuola Norm. Sup. Pisa}, 22, 1968, 265--274.

\bibitem{Vet} V\'etois, J., Multiple solutions for nonlinear elliptic equations on compact Riemannian manifolds, 
{\it Internat. J. Math.}, 18, 1071--1111, 2007.

\bibitem{Wit} Witten, E., A new proof of the positive energy theorem, 
{\it Comm. Math. Phys.}, 80, 381--402, 1981.

\end{thebibliography}
\end{document}